\documentclass[a4paper]{article}
\usepackage[english]{babel}
\usepackage[utf8]{inputenc}
\usepackage{amsmath,amsthm}
\usepackage{amsfonts,amssymb}

\usepackage{xspace}

\usepackage[colorlinks=true,linktocpage=true,linkcolor=blue,citecolor=red]{hyperref}
\usepackage{zref-clever}

\usepackage{enumitem}
\setenumerate[1]{label=(\arabic*)} 

\usepackage{mathtools}

\usepackage{tikz,tikz-cd} 
\usetikzlibrary{arrows}
\usetikzlibrary{calc} 

\newcommand{\Acal}{\mathcal{A}}

  \newcommand{\Rb}{\mathbb{R}}

\newcommand{\Dcal}{\mathcal{D}}

\newcommand{\Mcal}{\mathcal{M}}  
\newcommand{\Wcal}{\mathcal{W}}  
\newcommand{\Xcal}{\mathcal{X}}  
\newcommand{\Ccal}{\mathcal{C}}

\newcommand{\Ncal}{\mathcal{N}}  \newcommand{\Nb}{\mathbb{N}}

\newcommand{\I}{\textrm{I}}
\newcommand{\II}{\textrm{II}}

\swapnumbers

\theoremstyle{plain}
\newtheorem{theorem}{Theorem}[section]
\newtheorem{lemma}[theorem]{Lemma}
\newtheorem{prop}[theorem]{Proposition}
\newtheorem*{prop*}{Proposition}
\newtheorem{cor}[theorem]{Corollary}

\theoremstyle{definition}
\newtheorem{definition}[theorem]{Definition}

\newtheorem{remark}[theorem]{Remark}
\newtheorem{example}[theorem]{Example}
\newtheorem{assumption}[theorem]{Assumption}
\newtheorem{construction}[theorem]{Construction}
\newtheorem{notation}[theorem]{Notation}

\newcommand{\cref}[1]{\zcref[cap=true]{#1}}

\newcommand{\DeclareZCTheoremType}[5]{%
  \zcRefTypeSetup{#1}{%
    Name-sg = #2,
    name-sg = #3,
    Name-pl = #4,
    name-pl = #5,
  }%
  \AddToHook{env/#1/begin}{%
    \zcsetup{countertype={theorem=#1}}%
  }%
}

\zcRefTypeSetup{theorem}{
  Name-sg = Theorem,
  name-sg = theorem,
  Name-pl = Theorems,
  name-pl = theorems,
}

\DeclareZCTheoremType{lemma}
  {Lemma}{lemma}
  {Lemmas}{lemmas}

\DeclareZCTheoremType{prop}
  {Proposition}{proposition}
  {Propositions}{propositions}

\DeclareZCTheoremType{cor}
  {Corollary}{corollary}
  {Corollaries}{corollaries}

\DeclareZCTheoremType{conj}
  {Conjecture}{conjecture}
  {Conjectures}{conjectures}

\DeclareZCTheoremType{definition}
  {Definition}{definition}
  {Definitions}{definitions}

\DeclareZCTheoremType{remark}
  {Remark}{remark}
  {Remarks}{remarks}

\DeclareZCTheoremType{example}
  {Example}{example}
  {Examples}{examples}

\DeclareZCTheoremType{assumption}
  {Assumption}{assumption}
  {Assumptions}{assumptions}

\DeclareZCTheoremType{construction}
  {Construction}{construction}
  {Constructions}{constructions}

\DeclareZCTheoremType{notation}
  {Notation}{notation}
  {Notations}{notations}

\DeclareZCTheoremType{npar}
  {Paragraph}{paragraph}
  {Paragraphs}{paragraphs}

   \tikzset{ito/.style={hook}}
   
   \tikzset{ito/.style={-triangle}}

   \newcommand{\cto}{\rightarrowtail}
   \tikzset{cto/.style={>->}}
   
   \newcommand{\acyc}{\sim}

   \newcommand{\ano}{a}

   \tikzset{acto/.style={cto,"\ano"}}
   
   \tikzset{tcto/.style={cto,"\acyc"}}

   \tikzset{fto/.style={->>}}

   \tikzset{afto/.style={fto,"\ano"}}
   
   \tikzset{tfto/.style={fto,"\acyc"}}

\makeatletter
\def\namedlabel#1#2{\begingroup
    #2%
    \def\@currentlabel{#2}%
    \phantomsection\label{#1}\endgroup
}
\makeatother

\newcommand{\id}{\text{Id}}

\newcommand{\op}{\text{op}}

\newcommand{\sSet}{\text{sSet}}

\DeclareMathOperator*{\colim}{Colim}

\DeclareMathOperator{\Fun}{Fun}

\newcommand{\scheme}{collapsing scheme\xspace}
\newcommand{\schemes}{collapsing schemes\xspace}

\begin{document}

\pagestyle{plain}
\title{Rewriting and presentations of quasicategories}

\date{}

\author{Simon Henry}


\maketitle

\begin{abstract} We show that the methods of rewriting theory to establish coherence theorems can be applied at the level of quasicategories. More precisely, for any category $C$ which admits a presentation by a convergent rewrite system, we show that the corresponding weak $(\infty,1)$-category admits an Anick-Groves-Squier style presentation whose generators correspond to (some of) the critical branchings of the rewrite system. This follows entirely from reinterpreting K. Brown's simplicial proof of the usual homological Anick-Groves-Squier presentation in terms of the Joyal model structure instead of the Kan-Quillen model structure. We give several applications of this to the theory of quasicategories, including a relatively general coherence theorem for loop-free planar category theoretic diagrams and a new proof that pushout of Dwyer maps are homotopy pushouts.
\end{abstract}

\vspace{-0.2cm}

\renewcommand{\thefootnote}{\fnsymbol{footnote}} 
\footnotetext{\emph{email:} shenry2@uottawa.ca}
\renewcommand{\thefootnote}{\arabic{footnote}}


\tableofcontents

\section{Introduction}

The Anick-Groves-Squier theorem \cite{squier1987word},\cite{anick1986homology},\cite{groves2006rewriting},\cite{kobayashi1990complete} is a classical result in the theory of rewrite systems which, given a monoid $M$ presented by a convergent rewrite system (see section 2), produces a description of the homology of $M$ as the homology of a chain complex
\[ C_0 \leftarrow C_1 \leftarrow \dots\leftarrow  C_n \leftarrow \dots\]
where each $C_i$ is a free abelian group with an explicit set of generators. More precisely $C_0$ is free on the set of generators of the rewrite system, $C_1$ is free on its set of rewrite rules, $C_2$ is free on the set of critical pairs, and $C_n$ for $n \geqslant 2$ is free on a certain set of ``critical $(n-1)$-branchings\footnote{In practice, different references have used slightly different presentations with different sets of generators - we address this in \cref{sec:weak_vs_strong_criticallity}}''.

In \cite{brown1992geometry}, this was extended by K.S. Brown from a presentation of the homology of $M$ to what is essentially a cellular presentation of the classifying space of $M$, which allows us to recover the Anick-Groves-Squier theorem essentially by taking cellular homology.

More precisely, Brown's proof works by considering the simplicial nerve of $M$ and using what we would nowadays call Forman's discrete Morse theory \cite{forman1998morse}, \cite{forman2002user}, or Sean Moss' ``Anodyne presentations'' (\cite{moss2015another}) to collapse the simplicial nerve to a fairly small set of ``critical cells'' in $M$ that essentially corresponds to the critical branchings.

In the present note we observe that the exact same collapsing scheme given by Brown in \cite{brown1992geometry} can be applied more generally to categories instead of monoids, and more importantly already works at the level of quasicategories, and not just when passing to their classifying spaces. That is, Brown's construction can be used to promote a presentation of a category by a convergent rewrite system into a ``homotopy coherent presentation'' of the same category as an $(\infty,1)$-category, just in terms of the critical branchings of the rewrite system. More precisely our main theorem can be read as:

\begin{theorem}\label{main_th_intro}
  Let $\Ccal$ be a category presented by a convergent rewrite system. Then in the $\infty$-category of $\infty$-categories, $\Ccal$ can be realized as the colimit of a sequence
  \[\varnothing \to X_0 \to X_1 \to X_2 \to \dots \to X_n \to \dots \]
  Where for each $n$, $X_n$ is obtained from $X_{n-1}$ as a pushout
  \[\begin{tikzcd}
    C_n \times \partial \Delta_n \ar[r] \ar[d] \ar[dr,phantom,"\ulcorner"{description,very near end}] & C_n \times \Delta_n \ar[d] \\
    X_{n-1} \ar[r] & X_n   
    \end{tikzcd}\]
  where $C_n$ is a concretely defined subset\footnote{It is possible to take $C_n$ to be the set of all critical branchings, see the discussion in \cref{sec:weak_vs_strong_criticallity}, especially \cref{cor:weakly_crit_scheme}.} of the set of critical $(n-1)$-branchings of the rewrite system. 
\end{theorem}

To be more precise:

\begin{enumerate}
\item \cref{prop:Brown_collaps} gives a \scheme on $N(\Ccal)$ as a quasicategory.
\item \cref{rk:scheme_give_presentation} explains how such \schemes produce a homotopy presentation as in \cref{main_th_intro}.
\item \cref{sec:weak_vs_strong_criticallity} clarifies the connection between the critical cells (i.e. the generators) of the \scheme on $N(\Ccal)$ from \cref{prop:Brown_collaps} and critical branching of the rewrite system. It turns out that there are at least two different notions of critical branching used in the literature, which we refer to as ``weakly critical'' and ``strongly critical'' branchings, the \scheme from \cref{prop:Brown_collaps} uses the strongly critical branching, but it can be modified to use the weakly critical ones (see \cref{cor:weakly_crit_scheme}).
\end{enumerate}

In particular, the more precise version of the theorem produces an explicit combinatorial description of the set $C_n$ as well as of the ``boundary map'' $C_n \times \partial \Delta[n] \to X_{n-1}$ - mostly obtained from the definition of Brown's \scheme in \cref{def:Browm_collapsing} - though in practice the boundary maps can be complicated to unpack; the reader can have a look at \cref{ex:N2} and \cref{ex:N3} for examples of unpacking these computations in relatively simple cases.  These computations of the boundary maps are difficult to do in a systematic way. We consider giving a better description of these boundary maps (potentially using a different framework than quasicategories) to be an important problem. In the meantime we give applications to quasicategory theory by mostly focusing on the cases where there are few or even no critical branchings, or applications where we do not need to know what are the boundary maps.

\bigskip

We should also mention \cite{guiraud2012higher} which does essentially the same as our main result but in the context of strict $(\infty,1)$-categories (with strict inverses). The version of these results in the linear setting (for homological computation or presentations of dg-categories) are also well known, and linear rewriting is closely related to the theory of Gr\"obner basis. While these methods have been well exploited in the linear setting (for homological computation, or in the theory of dg-algebras and dg-categories), the fact that this can be done directly at the (non-linear) $\infty$-categorical level seems to have been overlooked - probably because quasicategories were not yet a mainstream object of study at the time Brown's paper \cite{brown1992geometry} was written. However, these constructions have many nice applications to the theory of $\infty$-categories. The second half of the paper is devoted to presenting some of these applications (\cref{sec:example_main}, \cref{sec:Dwyer_maps} and \cref{sec:example_diag}). We present both applications to concrete quasicategories as well as more abstract results applicable to general $(\infty,1)$-categories. 

\bigskip

For example, in \cref{sec:Dwyer_maps} we give a new proof of the fact that pushouts of Dwyer maps between categories are homotopy pushouts of $\infty$-categories (originally proved in \cite{hackney2024pushouts}). In \cref{sec:example_diag} we use these methods to obtain a coherence theorem for planar loop-free diagrams in quasicategories. Informally we justify that for planar loop-free diagrams there is no difference the $1$-category presented by making each inner face commutative and the $\infty$-category presented by making each inner face commutes up to an (invertible) 2-cell. This result can be used as a metatheorem to easily justify many diagram manipulations that we often do in $\infty$-category theory without always properly justifying them.

\section{Rewrite systems}

In this section we briefly recall the basics of rewriting theory. A rewrite system for a category $\Ccal$ is a presentation of $\Ccal$ by generators and relations, where the relations are ``oriented'', and called ``rewrite rules'':

\begin{definition}
  A \emph{rewrite system} for a category is the data of a graph $G$ and a set $R$ of pairs $(\gamma,\gamma')$ where $\gamma$ and $\gamma'$ are oriented paths in the graph with the same source and target point.
\end{definition}

Elements of $R$ are called rewrite rules and are represented as $\gamma \Rightarrow \gamma'$. Any rewrite system gives a presentation in the classical sense. Each rewrite system  $(G,R)$ present a category denoted $|G,R|$, where we first consider the free category $G^*$ on the graph $G$, i.e. the category of directed paths in $G$ with paths composition, and then we consider the quotient (as a category) which identify $\gamma$ and $\gamma'$ for each rule $\gamma \Rightarrow \gamma'$.

\begin{notation}\label{notation:reverse_order}
  When discussing rewrite systems, composites of paths will be represented in diagrammatic order instead of the usual category-theoretic order. i.e. in a graph
  \[\bullet \overset{a}\to \bullet \overset{b}\to \bullet \]
  the composite of $a$ and $b$ will be $ab$ and not $ba$. The reason for this is because latter we will use the nerve of these category and we want that face maps to be easy to write down as
  \[ d_i(\gamma_0,\dots,\gamma_n) = ( \gamma_0, \dots, \gamma_i \gamma_{i+1}, \dots, \gamma_n) \]
  which require to either use the diagramatic composition order or to write the simplicies in the reverse ordre.
\end{notation}

\begin{remark}
  Formally a rewrite system is the same as a $2$-computad (or $2$-polygraph) in the sense of \cite{street1976limits}, but we will not really make use of this point of view in the present note.
\end{remark}

\begin{example}\label{ex:squares}
  The graph
 \[ \begin{tikzcd}
    \bullet \ar[r,"f"] \ar[d,"h"swap] & \bullet \ar[d,"g"] \\
    \bullet \ar[r,"k"swap] & \bullet
  \end{tikzcd} \]
with the rewrite rule $hk \Rightarrow fg$ is a presentation of the category $[1] \times [1]$.
More generally, the graph
\[ \begin{tikzcd}
  x_{0,0} \ar[r,"a_{0,1}"] \ar[d,"b_{1,0}"] &   x_{0,1} \ar[r,"a_{0,2}"] \ar[d,"b_{1,1}"] & \dots \ar[r,"a_{0,n}"]  & x_{0,n} \ar[d,"b_{1,n}"] \\
  x_{1,0} \ar[r,"a_{1,1}"] \ar[d,"b_{2,0}"] &   x_{1,1} \ar[r,"a_{1,2}"] \ar[d,"b_{2,1}"] & \dots \ar[r,"a_{1,n}"] & x_{1,n} \ar[d,"b_{2,n}"] \\
  \vdots \ar[d,"b_{m,0}"] & \vdots \ar[d,"b_{m,1}"] & \ddots & \vdots \ar[d,"b_{m,n}"] \\
  x_{m,0} \ar[r,"a_{m,1}"] & x_{m,1} \ar[r,"a_{m,2}"] & \cdots \ar[r,"a_{m,n}"] &  x_{m,n}
\end{tikzcd}\]
with one rewrite rule per small square
\[ a_{i-1,j} b_{i,j}  \Rightarrow b_{i,j-1} a_{i,j}  \]
is a presentation of the category $[n] \times [m]$, where $[n]$ denotes the linear order $\{0 < 1 < \dots < n \}$.
\end{example}

\begin{example}\label{ex:Nk}
The graph with $k$ loops
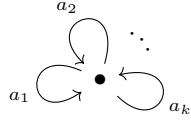
\[\begin{tikzpicture}
    \node (v) {%
        $\begin{tikzcd}[ampersand replacement=\&,rotate=60]
            \bullet
            \arrow[loop, in=165, out=105, distance=3em, "a_1"swap]
            \arrow[loop, in=85,  out=25,  distance=3em, "a_2"swap]
            \arrow[loop, in=-35, out=-95, distance=3em, "a_k"']
        \end{tikzcd}$
    };
    \foreach \ang in {25, 0, -25} {
        \node at ($(v)+(0.3,0.1)+({0.35*cos(\ang+45)},{0.35*sin(\ang+45)})$) {$\cdot$};
    }
\end{tikzpicture}\]
and all the rewrite rules
\[ a_j a_i \Rightarrow a_i a_j \text{ ($ \forall i <j$)}\]
is a presentation of the category $B\Nb^k$ with a single object and the monoid $\Nb^k$ as its endomorphisms.
\end{example}

\begin{notation}
  Given a path $p$ in $G$, by a \emph{subpath} of $p$ we mean a decomposition of $p$ as $p=u f v$. We will refer to $f$ as being the subpath by an abuse of notation, but we still consider the ``position'' of $f$ in $p$ (that is the data of $u$ and $v$) as being important. That is if $p = abcab$ then there are two different subpaths $ab$ in $p$, $p=(ab)cab$ and $p=abc(ab)$.

  We will call a \emph{subrule} of $p$ any subpath of $p$ which is also the domain of a rule.
\end{notation}

\begin{definition} Given a rewrite system $(G,R)$:
  \begin{enumerate}
  \item We say that a rewrite rule $r: \gamma \Rightarrow \gamma'$ applies to a path $p$ if $p$ has $\gamma$ as a subpath (i.e. subrule), i.e. $p = u \gamma v$.
  \item We say that $p \Rightarrow q$ is a one step reduction (along $r$) if $p = u \gamma v$ and $q= u \gamma' v$ for $r:\gamma \Rightarrow \gamma'$ one of our rewrite rules.
  \item We say that \emph{$p$ reduces to $q$}, and write $p \Rightarrow^* q$ if there is a finite sequence of one step reduction along various rewrite rules in $R$, $p \Rightarrow p_1 \Rightarrow \dots \Rightarrow p_n=q$.
  \item We say that a path $p$ is \emph{normal} if no rule can be applied to $p$ ($p$ has no subrules).
  \item We say that a rewrite system is \emph{confluent} if any time a path $p$ can be reduced in two different ways $p \Rightarrow^* q_1$ and $p \Rightarrow^* q_2$, there is path $p'$ such that $q_1 \Rightarrow p'$ and $q_2 \Rightarrow p'$.
  \item We say that a rewrite system is \emph{terminating} if there are no infinite sequences of reduction $p \Rightarrow p_1 \Rightarrow p_2 \Rightarrow \dots \Rightarrow p_n \Rightarrow \dots$.
  \item We say that a rewrite system is \emph{convergent} (sometimes also called \emph{complete}) if it is both confluent and terminating.
  \end{enumerate}
\end{definition}

It is easy to see from these definitions that:

\begin{lemma}
  \begin{enumerate}
  \item If $(G,R)$ is a terminating rewrite system then any path $\gamma$ reduces to at least one normal path.
  \item If $(G,R)$ is a confluent rewrite system then any path $\gamma$ reduces to at most one normal path.
  \item If $(G,R)$ is convergent then any path reduces to a unique normal path. This induces a bijection between the set of normal paths in $G$ and the set of arrows of the category presented by $(G,R)$.
  \end{enumerate}
\end{lemma}

\begin{example}
The rewrite system form \cref{ex:squares} and  \cref{ex:Nk} are all convergent.  
\end{example}

Critical pairs are paths in the graph to which two rules can be applied in a way that is both overlapping and minimal. For example, in the system of \cref{ex:Nk}, $(a_3a_2)a_1 = a_3(a_2a_1)$ is a critical pair as we apply both rules $a_3a_2 \Rightarrow a_2 a_3$ and $a_2 a_1 \Rightarrow a_1 a_2$, they overlap on $a_2$ and this configuration is minimal. But $a_2a_1a_2a_1$ is not critical because the two rules are not overlapping and $a_3a_2a_1a_1$ is not critical because the configuration is not minimal as the last $a_1$ can be removed and both rules can still be applied. More explicitly:

\begin{definition}
  If $(G,R)$ is a rewrite system. A \emph{critical pair} is a pair of rules $R_1: \gamma_1 \Rightarrow \gamma'_1$ and $R_2: \gamma_2 \Rightarrow \gamma'_2$ and decompositions of the form $\gamma_1 = va$ and $\gamma_2=aw$ with $a$ a non-trivial path.
\end{definition}

We generally denote such a critical pair by ``$vaw$'' but it is sometimes necessary to clarify which rules can be applied by writing for example $(va)w = v(aw)$.

A central result in the field is:

\begin{lemma}[Newman]\label{lem:newman}
  A terminating rewrite system $(G,R)$ is confluent if and only if for each critical pair whose two one-step reductions are $\psi \Rightarrow \gamma$ and $\psi \Rightarrow \gamma'$, there is a path $\gamma''$ such that $\gamma \Rightarrow^* \gamma''$ and $\gamma' \Rightarrow^* \gamma''$. 
\end{lemma}

See \cite{newman1942theories} for the original reference and \cite{huet1980confluent} for a simpler proof. We say that a rewrite system is locally confluent if it satisfies the assumption of Newman's lemma on critical pairs (but might not be terminating).

We conclude by the following example which show that every category can be canonically presented by a convergent rewrite system:

\begin{example}\label{ex:std_rewrite} If $\Ccal$ is any category, it can be presented by the ``standard rewrite system'':

  \begin{itemize}
  \item The graph $G$ is the underlying graph of $\Ccal$ (one generator for each arrow if $\Ccal$). For $f$ an arrow in $\Ccal$, we write $[f]$ for the corresponding edge of $G$.
  \item For each composable pairs of arrow of $f,g$ in $\Ccal$, we have a rule $[f][g] \Rightarrow [fg]$.
   \item We also have a rewrite rule $[\id_x] \Rightarrow 1_x$ for each identity arrow $\id_x$.
  \end{itemize}

  The system is clearly terminating because each rule reduces the length of the word and we can check confluence using Newman's lemma, as the critical pairs are:

  \begin{enumerate}
  \item $[f][g][k] \Rightarrow [f][gk] , [fg][k]$ and both reduce to $[fgk]$.
  \item $[f][\id_x] \Rightarrow [f \id_x]=[f], [f]$ which are already the same.
  \item $[\id_x][f]$ which behaves exactly the previous one.
  \end{enumerate}

Hence this is a terminating rewrite system. The normal cells are the trivial paths $1_x$ and the $[f]$ for $f$ a non-identity arrow of $\Ccal$. So this is a presentation of $\Ccal$. 
\end{example}

\begin{remark}
  While the previous example shows that for a category, having a presentation by a convergent rewrite system is not a special property at all, having smaller rewrite system than this one is a special property. Notably having a presentation by a finite convergent rewrite system is a special property amongst finitely presented categories: One of the important consequences of the classical Anick-Groves-Squier theorem is that if a monoid $\Mcal$ (and more generally a category $\Ccal$) has a presentation by a finite rewrite system, then all the homology\footnote{These can be defined as the homology group of the simplicial nerve of $\Mcal$.} groups of $\Mcal$ are finitely generated. This is not true in general for a finitely presented monoid or category.
\end{remark}

\section{Quasicategorical \scheme}

In this section we introduce an analogue of Forman's Discrete Morse theory \cite{forman2002user}, or Sean Moss' anodyne extension \cite{moss2015another} adapted to the Joyal model structure on quasicategories. Essentially these are structures on simplicial sets that allow us to keep track of how it could have been built from a small number of ``critical cells'', together with pushouts of inner horn inclusions. The construction by Brown of such a structure on $N|G,R|$ for $(G,R)$ a convergent rewrite system is how the main result of this paper is established.

\begin{definition}\label{def:qcat_scheme} Let $i:A \cto X$ be a cofibration of simplicial sets. A quasicategorical \emph{\scheme} on $i$ is the data of:

  \begin{itemize}
  \item A partition of the set of non-degenerated cells of $X$ not in $A$ into three classes: the  type $\I$ cells, the type $\II$ cells and the critical cells.
  \item A bijection $c$ from the set of type $\I$ cells to the set of type $\II$ cells.
  \end{itemize}

  Such that:

  \begin{enumerate}
  \item For each type $\I$ cell $x$, there is a unique $i$ such that $c(x) = d_i x$, and moreover $0<i<n$.
  \item The order relation on the set of type $\II$ cells of $X$ generated by $x \leqslant y$ if $x = d_j c^{-1}(y)$ is well-founded.
  \end{enumerate}

  A \scheme of a simplicial set $X$ is a \scheme of the cofibration $\varnothing \to X$.
\end{definition}

\begin{remark}
  Every cofibration has a ``trivial'' \scheme where every cell is critical. A \scheme on $A \cto B$ is the same as a \scheme on $B$ such that all cell in $A$ are critical.
\end{remark}

\begin{remark}
  In Brown's terminology from \cite{brown1992geometry}, critical cells are called \emph{essentials}, type $\I$ cells are called \emph{collapsible} and type $\II$ cells are called \emph{redundant}. We use the terminology coming from Moss' \cite{moss2015another} instead.
\end{remark}

The main reason for this definition is that this is the structure that exactly captures constructing simplicial sets by iterated pushout of inner horn inclusions and boundary inclusions:

\begin{prop}
  A cofibration $A \cto B$ admits a quasicategorical \scheme with given set of critical cells, type $\I$ and type $\II$ cells if and only if it is a transfinite composition of pushouts of maps:
  \[ \partial \Delta [n] \cto \Delta[n]  \qquad \text{ and }\qquad \Lambda^k[n] \cto \Delta[n]\]
  such that
  \begin{enumerate}
  \item The critical cells of $B$ are exactly the non-degenerate cells added by pushouts of $\partial \Delta[n] \cto \Delta[n]$.
  \item The type $\I$ cells of $B$ are exactly the cells added as the $n$-dimensional cell of $\Delta[n]$ in a pushout of $\Lambda^k[n] \cto \Delta[n]$.
   \item The type $\II$ cells of $B$ are exactly the cells added as the $k$-th $(n-1)$-dimensional face of $\Delta[n]$ in a pushout of $\Lambda^k[n] \cto \Delta[n]$
  \end{enumerate}
\end{prop}

\begin{proof}
  See for example Section 2 of \cite{moss2015another}. The statement there use arbitrary Horn inclusion instead of inner horn inclusions and has no critical cells, but the proof is exactly the same.
\end{proof}

For example, a cofibration with a \scheme with no critical cells is inner anodyne, and every inner anodyne cofibration is a retract of a cofibration with a \scheme having no critical cells.

In some cases it will be useful to ``restrict'' \schemes:

\begin{definition} Given $A \cto B$ a cofibration. A \scheme on $B$ is said to restrict to $A$ if the bijection between type $\I$ and type $\II$ cells as well as its inverses preserves $A$.\end{definition}

The following is immediate:

\begin{lemma}\label{lemma:restrict_scheme}
  If a \scheme on $B$ restricts to a subsimplicial set $A \overset{i}{\cto} B$ then:
  \begin{enumerate}
  \item We have a \scheme on $A$ which is the restriction of the one on $B$.
  \item We have a \scheme on the cofibration $i$ which is the restriction of the one on $B$.
  \end{enumerate}  
Where by ``the restriction of the one on $B$'' we mean that a cell is type $\I$, resp. type $\II$, resp. critical if and only if it is in $B$, and the bijection $c$ is the restriction of the one on $B$.
\end{lemma}

In particular:

\begin{cor}
  Let $B$ be equipped with a \scheme that restricts to a simplicial subset $A \subset B$ and $A$ contains all critical cells, then $A \to B$ is inner anodyne.
\end{cor}

\begin{lemma}\label{lemma:Po_of_scheme}
  If $i:A \cto B$ is a cofibration equipped with a \scheme. Then any pushout of $i$
  \[\begin{tikzcd}
      A \ar[d,cto,"i"] \ar[r,"f"] & C \ar[d,cto,"j"] \\
      B \ar[r,"g"] & D
    \end{tikzcd}\]
is equipped with a \scheme where the type of a cell in $D-C$ is given by the type of its unique pre-image by $g$ and the pairing function is constructed in a similar way.
\end{lemma}
\begin{proof}
  The fact that this is a \scheme on $j$ is immediate as the set of cells in $D$ not in $C$ is in bijection with the cells of $B$ not in $A$ and the definition of \scheme only depends on these.
\end{proof}

\begin{prop}\label{rk:scheme_give_presentation}
  Given a \scheme on $A \cto B$ we can define $A^{(n)} \subseteq B$ to be collection of:
  \begin{enumerate}
  \item All cells of $A$.
  \item All critical cell of dimension $\leqslant n$.
  \item All type $\II$ cell of dimension $<n$.
  \item All Type $\I$ cell of dimension $\leqslant n$.
  \item All degeneracies of cells mentioned above.
  \end{enumerate}

then each $A^{(n)}$ is a simplicial subset of $B$ and $B= \colim A^{(n)}$, and each $A^{(n-1)} \to A^{(n)}$ is a homotopy pushout of the form

    \[\begin{tikzcd}
    C_n \times \partial \Delta_n \ar[r,cto] \ar[d] \ar[dr,phantom,"\ulcorner"{description,very near end}] & C_n \times \Delta_n \ar[d] \\
    A^{(n-1)} \ar[r,cto] & A^{(n)}   
  \end{tikzcd}\]
where $C_n$ is the set of critical $n$-cells of $B$.
\end{prop}

 In particular, the map $A \to B$ is a (homotopy) colimit of a sequence of cofibration
\[A \cto A^{(0)} \cto A^{(1)} \cto \dots \cto A^{(n)} \cto \dots \cto B\]
exactly as in \cref{main_th_intro}. So in order to prove \cref{main_th_intro} all we need is to equip the cofibration $\varnothing \cto N \Ccal$ with a \scheme whose critical cells correspond to critical branchings of the rewrite system presenting $\Ccal$. Which is essentially what we will do in \cref{def:Browm_collapsing} and \cref{prop:Brown_collaps}, see also \cref{sec:weak_vs_strong_criticallity} for the connection between the critical cell of the \scheme we will introduce in \cref{def:Browm_collapsing} and critical branching.

\begin{proof}
It is clear that $A^{(n)}$ is a simplicial subset: For each cell of $A^{(n)}$ its degeneracies are also in $A^{(n)}$ by definition, All the non-degenerate cells of $A^{(n)}$ are either in $A$ or of dimension $\leqslant n$, so their faces are all either in $A$ or of dimension $<n$ and both cells in $A$ and cells of dimension $<n$ are in $A^{(n)}$, so all faces of non-degenerate cells in $A^{(n)}$ are also in $A^{(n)}$. As all degenerate cells in $A^{(n)}$ are by definition degeneracies of a non-degenerate cell in $A^{(n)}$, all their faces are also in $A^{(n)}$. The fact that $B$ is the colimit of the $A^{(n)}$ is clear as directed union are colimits in the category of simplicial sets. It is also immediate that the \scheme on $A \to B$ restricts to a \scheme on $A^{(n-1)} \to A^{(n)}$ whose critical cells are exactly the critical $n$-cells of $A \to B$, that is $A^{(n-1)} \to A^{(n)}$ is obtained by gluing the critical $n$-cells, i.e. a pushout like the one in the proposition, and then attaching all remaining type $\I$ and $\II$ cells by pushouts of horn inclusions - which does not change the resulting object up to homotopy equivalence, and hence we have the homotopy pushout claimed in the proposition.
\end{proof}

Finally, \scheme are also convenient to check that some square are homotopy pushouts:

\begin{lemma}\label{lem:bij_on_crit}
  Let $A \cto B$ and $A \cto B'$ be two cofibrations equipped with \scheme. Assume there is a commutative triangle
  \[\begin{tikzcd}
    A \ar[d,cto] \ar[dr,cto] \\
    B \ar[r,"f"swap] & B'
  \end{tikzcd}\]
   where $f$ induces a bijection between the critical cells of $B$ and the critical cells of $B'$. Then $f$ is a weak equivalence. 
\end{lemma}

\begin{proof}
  Consider $A^{(n+)}$ to be defined similarly to $A^{(n)}$ in \cref{rk:scheme_give_presentation} but we add to it the type $\II$ cells of dimension $n$ and the type $\I$ cells of dimension $n+1$, as well as all the degeneracies. Then a similar argument shows that $A^{(n+)}$ is also a subsimplicial set of $B$, and we have a sequence of inclusions $A^{(n)} \subseteq A^{(n+)} \subseteq A^{(n+1)} \subseteq A^{(n+1+)}$ with the \scheme on $B$ restricting to all these inclusions and showing that each $A^{(n)} \to A^{(n+)}$ is inner anodyne, while $A^{(n+)} \to A^{(n+1)}$ is just adding freely the critical $(n+1)$-cells.

  Writing $A'^{(n+)}$ for the similar objects constructed in $B'$, we immediately see that the map $f$ sends each $A^{(n)}$ to $A'^{(n+)}$ as $A'^{(n+)}$ contains all cells of dimension $\leqslant n$ (and all cells in $A$) and the non-degenerate cells of $A^{(n)}$ are all either in $A$ or of dimension $\leqslant n$. One can then show by induction that $f$ induces a homotopy equivalence between $A^{(n)}$ and $A'^{(n+)}$ because both are defined from $A^{(n-1)}$ and $A'^{(n-1+)}$ by taking the same homotopy pushout. Taking the countable union (which is a homotopy colimit, as a colimit of a sequence of cofibrations) gives us that $f$ induces an equivalence between $B = \bigcup A^{(n)}$ and $B' = \bigcup A'^{(n+)}$.
\end{proof}

\begin{cor}\label{cor:h_pushout_from_rw}
  Given a square in $\sSet$
  \[\begin{tikzcd}
      A \ar[d,cto,"i"] \ar[r] & C \ar[d,cto,"j"] \\
      B \ar[r] & D
    \end{tikzcd}\]
Assume that $i$ and $j$ are equipped with \schemes such that the map $B \to D$ induces a bijection between the critical cells of $B$ and $D$. Then the square is a homotopy pushout. \end{cor}

\begin{proof} By \cref{lemma:Po_of_scheme} we have a \scheme on $C \cto B \coprod_A C$. And the map $ B \coprod_A C \to D$ is a bijection on critical cells, so by \cref{lem:bij_on_crit} it is a weak equivalence. As $ B \coprod_A C$ is a homotopy pushout (pushout along a cofibration) this proves the result.
\end{proof}

\section{Brown's \scheme}

In this section, we fix $\Ccal=|G,R|$, a $1$-category presented by a convergent rewrite system $(G,R)$, and we will construct a quasicategorical \scheme on $N(\Ccal)$. We mostly follows Brown's \cite{brown1992geometry}, with the exception that Brown only do the case where $\Ccal$ is a monoid and only claim to build a scheme in the sense of Kan-Quillen model structure, but the same works for the Joyal model structure as well.

The arrows of $\Ccal$ can hence be represented as normal path in $(G,R)$ and composition in $\Ccal$ corresponds to composition of path followed by normalization.

\begin{remark}
  The nerve of $N \Ccal$ hence has for $n$-simplexes the sequences
  \[ (\omega_0, \dots,\omega_{n-1}) \]
  of composable paths in normal form. Degeneracy maps correspond to inserting an empty path, inner face maps take the composite of two successive paths (remove the comma) and apply the normalization procedure to the composite. Outer faces remove $\omega_{n-1}$ or $\omega_0$ from the list.

A cell $(\omega_0, \dots,\omega_{n-1})$ is non-degenerate if and only if none of the $\omega_i$ is the empty word.
\end{remark}

We can now present Brown's \scheme from \cite{brown1992geometry}, which is the central object of this note:

\begin{construction}[Brown \cite{brown1992geometry}]
\label{def:Browm_collapsing}
  A non-degenerate cell $(\omega_0,\dots,\omega_{n-1})$ of $N\Ccal$ is said to be critical if it satisfies the following three conditions:
  \begin{enumerate}
  \item[$(C0)$] $\omega_0$ is a generator (a single letter word).
  \item[$(C1_i)$] Each composite $\omega_{i-1} \omega_i$ is reducible (not normal)
  \item[$(C2_i)$] All strict initial segments of $\omega_{i-1} \omega_i$ are  normal.
  \end{enumerate}

  A non-degenerate, non-critical cell is said to be of type $\II$ if either:
  \begin{itemize}
  \item $\omega_0$ is not a generator, or
  \item at the smallest $i$ where either conditions $(C1_i)$ and $(C2_i)$ fails, it is condition $(C2_i)$ fails.
  \end{itemize}
  
  And finally, it is type $\I$ otherwise, i.e.:

  \begin{itemize}
  \item $\omega_0$ is a generator and,
  \item at the smallest $i$ where either conditions $(C1_i)$ and $(C2_i)$ fails, it is condition $(C1_i)$ that fails\footnote{Note that $(C1_i)$ and $(C2_i)$ cannot fail at the same time as if some initial segment of $\omega_{i-1}\omega_i$ is reducible, then so is $\omega_{i-1}\omega_i$.}. 
  \end{itemize}

  Finally, we define the pairing functions $c$ on Type $\I$ cells as
  \[ c(\omega_0,\dots,\omega_{n-1}) = (\omega_0, \dots, \omega_{i-1} \omega_{i}, \dots ,\omega_n)\]
  where $i$ is the first index where condition $(C1_i)$ fails. Importantly, by definition of condition $(C1_i)$, this composite $\omega_{i-1}\omega_i$ is already in normal form.
\end{construction}

\begin{example}
  For an arbitrary category $\Ccal$, we have seen in \cref{ex:std_rewrite} that we can consider the ``standard rewrite system''. The cells of $N(\Ccal)$ are then just sequences $(f_0,\dots,f_{n-1})$ of composable arrows of $\Ccal$. Non-degenerate cells are the sequences of non-identity cells. Expressions of these as normal form paths are of the form $(p_0,\dots,p_{n-1})$ where each $p_i$ is either $[f]$ for $f$ a non-identity arrow or $p_i$ is a trivial path $1_x$. So non-degenerate cells are of the form $([f_0],\dots,[f_{n-1}])$ and it turns out they are all critical. So in the special case of this standard rewrite system, Brown's \scheme is not very interesting: it just makes every cell critical.
\end{example}

\begin{prop}[Brown \cite{brown1992geometry}] \label{prop:Brown_collaps}
 For any category $\Ccal$ presented by a convergent rewrite system, the structure from \cref{def:Browm_collapsing} is a \scheme on $\varnothing \cto N(\Ccal)$.
\end{prop}

\begin{proof}
  The proof from \cite{brown1992geometry} can be applied with essentially no changes. For completeness we include the details of the proof:

\textbf{The pairing $c$ is a bijection between type $\I$ and type $\II$ cells:}

Recall from \cref{def:Browm_collapsing} that for any type $\I$ cell $x=(\omega_0,\dots,\omega_n)$,
we define
\[ c(x) = (\omega_0,\dots,\omega_{i-1}\omega_i, \dots,\omega_n)\]
where $i$ is the smallest index such that $\omega_{i-1}\omega_i$ is irreducible. 

First observe that $c(\omega_0,\dots,\omega_{n-1})$ is always a type $\II$ cell:
\begin{itemize}
\item either $i=1$ in which case $(\omega_0, \dots, \omega_{i-1} \omega_i, \dots ,\omega_{n-1})$ will start with $(\omega_0 \omega_1,\dots)$ and hence the first path will not be a generator,
\item or $i>1$ in which case $\omega_{i-2}\omega_{i-1} \omega_{i}$ will have an initial segment $\omega_{i-2}\omega_{i-1}$ which is reducible, while $(\omega_0,\dots,\omega_{i-2})$ is a critical cell.
\end{itemize}

Also note that in both cases, we can reconstruct $x$ from $c(x)$ by this observation: if the first word of $c(x)$ is not a generator then we are in the first case and we get $x$ by separating the first letter back. If $c(x)=(v_0,\dots,v_{n-1})$ is type $\II$ because the first failure of \cref{def:Browm_collapsing} is that $v_{j-1}v_{j}$ has a reducible initial segment, then $x$ can be reconstructed by separating this first initial reducible segment. Indeed, in $c(x) = (\omega_0,\dots,\omega_{i-1}\omega_i, \dots,\omega_n)$ as $\omega_{i-2}\omega_{i-1}$ has no reducible strict initial segment, $\omega_{i-2}\omega_{i-1}$ is exactly the first reducible initial segment of $\omega_{i-2}\omega_{i-1}\omega_i$.

Finally, the same discussion also shows that if we start from a type $\II$ cell $x$, and do the same construction as described above we obtain the unique type $\I$ cell $y$ such that $x=c(y)$.

\textbf{There is a unique $i$ such that $c(y)=d_i(y)$.}

Indeed for any type $\I$ cell $y$, $c(y)$ is by definition obtained by ``removing a comma'' from $y$, so it is $d_i(y)$ for some inner $0<i<n$. Taking $d_j(y)$ for $0\leqslant j< i$ will either remove the first letter, or result in the application of a reduction rule. So it cannot produce the same $c(y)$. Taking $d_j(y)$ for $j>i$ will result in the application of a reduction rule, or in a cell starting with $(\omega_0, \dots,\omega_{i-1},\dots)$ which cannot be equal to $c(y)$ as $\omega_{i-1} \omega_i \neq \omega_{i-1}$ (remember that as the cells under consideration are all non-degenerate, we cannot have $\omega_{i} = 1$).

\textbf{The \scheme is well-founded.}

The relation $x \leqslant y$ as defined in \cref{def:qcat_scheme}, in the case of Brown's collapsing scheme \cref{def:Browm_collapsing}, means that $y= (\omega_0,\dots,\omega_n)$ is a type $\II$ cell, and its corresponding type $\I$ cell $c^{-1}y$ is obtained by ``splitting'' one of the $\omega_i$. Either $i=0$, in which case $\omega_0$ is not a single letter and we write $\omega_0=uv$ where $u$ is its first letter, or we consider the first $\omega_i$ such that $\omega_{i-1} \omega_i$ has a reducible initial segment, and we write $\omega_i=uv$ so that $\omega_{i-1}u$ is the smallest such reducible initial segment. In both cases we can write $c^{-1}= (\omega_0,\dots,\omega_{i-1},u,v,\omega_{i+1},\dots,\omega_n)$. Then $x$ is any face of $c^{-1}y$. For any cell $y =(\omega_0,\dots,\omega_n)$ we denote by $|y|$ the total composite $\omega_0 \dots \omega_n$ as a (possibly) non-reduced path. Note that, in our situation above, $|y| = |c^{-1}(y)|$.

There are a few different things that can happen:

\begin{enumerate}
\item If $x = d_0(c^{-1} y )$ or $x=d_{n+1}(c^{-1} y)$, this removes the first or last $\omega_i$ from $y$. In this case $|x|$ is a strict subpath of $|y|$.
\item If $x = d_jc^{-1} y$ for $j\leqslant i$, then a reduction always occurs in the new composite that appears when removing the comma, so $|y| \Rightarrow^* |x|$ is a non-trivial reduction.
\item If $j=i+1$, then $x=y$ and we can ignore this case, as well-foundedness of $\leqslant$ implicitly means well-foundedness of the corresponding strict order relation $<$.
\item $i+1 < j <n+1$, then $|x|$ can be either $|y|$ or obtained from $|y|$ by some reduction, but the cell $x$ is ``closer to being critical'', in the sense that the largest initial part of $y$ that is critical was $(\omega_0,\omega_1,\dots,\omega_{i-1})$ and for $x$ it is now $(\omega_0,\dots,\omega_{i-1},u)$.
\end{enumerate}

Assuming, toward a contradiction, that there is an infinite sequence of such steps, there can only be a finite number of reductions as in the last case in a row, as the dimension of the cell is preserved, so the length of the critical part cannot increase forever; every other reduction replaces $|y|$ with either a subpath or a strict reduction $|x|$. So any infinite sequence of such steps will always produce (by skipping all steps corresponding to the last case) an infinite sequence of paths $y_1,\dots,y_n,\dots$ such that each $y_{i+1}$ is either a subpath or a reduction of $y_i$. As there can only be a finite number of ``subpath'' steps in a row, this in turn produces an infinite sequence of reductions $y_0 \Rightarrow y_1 \Rightarrow \dots \Rightarrow y_n \Rightarrow \dots$, by skipping all the subpath steps and adding back to all subsequent terms all the parts of the path that have been removed. This would contradict the assumption that the rewrite system is terminating. 
\end{proof}

\begin{remark}
  Combining \cref{prop:Brown_collaps} and \cref{rk:scheme_give_presentation} this proves \cref{main_th_intro} - with the $C_n$ being the set of critical cells. The relation between critical cells and critical branching needed to really complete the proof is discussed in the next section.
\end{remark}

We conclude this section by the following easy lemma that will be useful later:

\begin{lemma} \label{lem:extension_of_rw}
Let $A$ and $B$ be two categories presented by convergent rewrite systems $(G_A,R_A)$ and $(G_B,R_B)$ such that $G_A \subseteq G_B$ and $R_A \subseteq R_B$. Assume that each normal path in $(G_A,R_A)$ is still normal in $(G_B,R_B)$. Then the induced functor $F:A \to B$ gives a cofibration $N(A) \to N(B)$ and Brown's \scheme on $N(B)$ restricts to $N(A)$. The \scheme it induces on $N(A)$ is Brown's \scheme for $(G_A,R_A)$, and it also induces a \scheme on the cofibration $N(A) \cto N(B)$.
\end{lemma}

\begin{proof}
  Faithfulness of $F$ is immediate: as normal paths in $(G_A,R_A)$ are assumed to still be normal in $(G_B,R_B)$, they cannot reduce to the same thing without being already equal as paths. This immediately implies that $N(A) \to N(B)$ is a cofibration (i.e. a monomorphism).

    The key observation is then simply that by construction, the pairing between type $\I$ and type $\II$ cells does not change the list of symbols that appear in a cell, as a cell will be in $N(A)$ if and only if, (after normalization), all the symbols that appear in it are in $G_A$, hence the pairing preserves $N(A)$, and we conclude by \cref{lemma:restrict_scheme}.
\end{proof}

\section{Relation between critical cell and critical branching}
\label{sec:weak_vs_strong_criticallity}

In this section we relate the critical cells of Brown's \scheme with the ``critical $n$-branchings''. The only difficulty here is that it turns out there are two non-equivalent definitions of critical branching in the literature, which we will refer to as ``weakly critical'' and ``strongly critical''. The critical $n$-cells of Brown's \scheme as introduced in \cref{def:Browm_collapsing} correspond to what we call the \emph{strongly critical $(n-1)$-branching}, while the weakly critical branching is probably the more natural notion, which is used for example in \cite{guiraud2012higher}. We also show that it is possible to modify Brown's \scheme so that the critical $n$-cells correspond to the weakly critical $(n-1)$-branchings.

Our discussion of ``weakly critical'' and ``strongly critical'' branching is exactly the discussion in \cite{groves2006rewriting} of the comparison between Groves ``critical $n$-stars'', which corresponds to our weakly critical $n$-branching and what he calls the ``special critical $n$-stars'' or ``minimal non-overlapping critical $n$-stars'', which corresponds to our strongly critical $n$-branching and to Anick's $n$-chains from \cite{anick1986homology}. So we expect the results in this section might not be new to people already familiar with rewriting, but the author thought this might be useful to people interested in applications to $\infty$-category theory - and it will come up in some of our examples later.

In order to discuss this, we need to first do some ``clean-up'' in our rewrite system, and assume it is \emph{reduced} in the following sense:

\begin{definition}\label{def:reduced}
  A rewrite system $(G,R)$ is said to be reduced if:
  \begin{enumerate}
  \item If $p:\gamma \Rightarrow \gamma'$ is a rules in $R$, then no other rules can be applied to $\gamma$.  
  \item No rules of $R$ has a single generator as its domain.
  \end{enumerate}
\end{definition}

\begin{remark}
  If a convergent rewrite system is not reduced, we can make it both reduce and convergent without changing the category it presents by:
  \begin{enumerate}
  \item For any rule $p:\gamma \Rightarrow \gamma'$ such that another rule $p'$ can be applied to $\gamma$, we can remove $\gamma$ from the rules. Removing a rule preserve termination, and this change does not affect the set of normal paths, so it also preserves confluence and the presented category.
   \item For any rule $p:x \Rightarrow \gamma$ such that $x$ is a letter, then after the previous step $x$ cannot appear in the domain of any other rules. We can then remove $x$ from $G$ and replace $x$ by $\gamma$ in the target of any other rules. For any path in which $x$ appear, we can replace $x$ by $\gamma$ without changing its image in the presented category.
   \end{enumerate}

   In the rest of this section, we will always assume that our rewrite system is reduced.
  
\end{remark}

\begin{construction}\label{cstr:order_on_subrules}
  If a rewrite system $(G,R)$ is reduced, $p$ is a path in $G$ and $f_1,f_2$ two subrules of $p$, then the following are equivalents:
  \begin{enumerate}
  \item The start position of $f_1$ is strictly before the start position of $f_2$.
  \item The end position of $f_1$ is strictly before the end position of $f_2$.
  \end{enumerate}
  Moreover, if $f_1 \neq f_2$ then neither their start position or end position can be the same. Indeed, any configuration where this is not the case corresponds to an inclusion of $f_1$ as a subpath of $f_2$, or an inclusion of $f_2$ as a subpath of $f_1$, which would contradict the assumption that $(G,R)$ is reduced.

  We write $f_1 < f_2$ if the start (equivalently end) position of $f_1$ is strictly before the start (equivalently end) position of $f_2$. For any two such subpaths (that are domain of rules) we have $f_1 < f_2$ or $f_1=f_2$ or $f_2 < f_1$.

  If $f_1 < f_2$ we say that $f_1$ and $f_2$ overlap is their positions overlap, i.e. if the start position of $f_2$ is before or equal to the end position of $f_1$.

\end{construction}

\begin{definition}\label{def:weakly_crit_branching} Let $(G,R)$ be a rewrite system.
  \begin{enumerate}
  \item An $n$-Branching of a rewrite system is a path $p$ together with $n$ chosen subrules $(f_1 < \dots < f_n)$ of $p$.
  \item If $(p; f_1,\dots,f_n)$ and $(q; g_1,\dots,g_m)$ are an $n$-branching and an $m$-branching with $p$ and $q$ composable, then $(pq;f_1,\dots,f_n,g_1,\dots,g_m)$ is an $(n+m)$ branching.
   \item An $n$-branching $w$ is said to be ``weakly critical'' if it cannot be written as a composite of two branchings $w=pq$ as in the previous point unless $p=\varnothing$ or $q = \varnothing$.
  \end{enumerate}
\end{definition}

It is immediate that:

\begin{lemma}\label{lem:weakly_crit_branching}
  In a reduced rewrite system, an $n$-branching $(p;f_1,\dots,f_n)$ is weakly critical if and only if:
  \begin{enumerate}
  \item The start position of $f_1$ is the start of $p$.
  \item The end position of $f_n$ is the end of $p$.
  \item each $f_i$ overlap with $f_{i+1}$.
  \end{enumerate}    
\end{lemma}

\begin{proof}
  Indeed, a decomposition of $p$ as a composite $p=p_1p_2$ of a $k$ branching and an $(n-k)$ branching with $p_1$ and $p_2$ non-trivial, would either correspond to some letters before the first rule if $k=0$, a few letters after the last rule if $n = k$, or an absence of overlap between rule $f_k$ and the rule $f_{k+1}$ if $0<k<n$. Here we are using that as the rewrite system is reduced, it follows from the observation made in \cref{cstr:order_on_subrules} that if $f_k$ and $f_{k+1}$ are non-overlapping then no rule $f_i$ for $i \leqslant k$ overlaps with any $f_j$ for $j \geqslant k+1$.
\end{proof}

\begin{construction}\label{cstr:cell_from_wcrit_branch}
  Given a weakly $n$-branching $(p;f_1,\dots,f_n)$ we can associate to it a unique $n+1$-cell $c(p;f_1,\dots,f_n) = (\omega_0,\dots,\omega_n)$ in $N|G,R|$ such that:
  \begin{enumerate}
  \item $p=\omega_0\dots\omega_n$
  \item $\omega_0$ is a single generator.
  \item the end point of the rule $f_i$ coincide with the end point of the subword $\omega_i$.
  \end{enumerate}

Indeed, the condition above clearly define the end point of each $\omega_i$, and hence as well their start point, as subpath of $p$. The end point of $f_n$, and hence of $\omega_n$ is indeed the end point of $p$, so the first condition is satisfied. we only need the lemma below to show that this indeed make sense:
  
\end{construction}

\begin{lemma}\label{lem:weakly_crit_cell_normal} Given a weakly critical $n$-branching $(p;f_1,\dots,f_n)$ in a reduced rewrite system,  all the $\omega_i$ in the previous construction are normal paths.\end{lemma}

\begin{proof} $\omega_0$ is normal because part of the assumption that $(G,R)$ is reduced is that all generators are normal. $\omega_1$ is normal because the rule $f_1$ ends at the end of $\omega_1$ and starts at the beginning of $p$, so $f_1 = \omega_0\omega_1$. Any subrule of $\omega_1$ would be a strict subrule of $f_1$, which is impossible as the rewrite system is reduced. For $i>1$, if $\omega_i$ contained a subrule $r$, then, when considered as a subrule of $p$,  we would have $r \leqslant f_i$ and the start of $r$ would be after the end of the rule $f_{i-1}$ (as $r$ is a subpath of $\omega_i)$. But as we are assuming that $(p;f_1,\dots,f_n)$ is weakly critical we have that $f_i$ starts before (or equal to) the end of $f_{i-1}$, so this would force $r$ to be a subrule of $f_i$, which is impossible as $(G,R)$ is reduced.
\end{proof}

\begin{example} \label{ex:en=>1_weaklycrit}
  Consider the rewrite system where $G$ has a single arrow $e: \bullet \to \bullet$ and a single rewrite rule $e^n \Rightarrow 1$. Then it is easy to see that for each $k$, there is a unique critical $k$-cell $c_k$ in Brown \scheme:

  \[ c_0 =() \qquad c_1=(e) \qquad c_2=(e,e^{n-1}) \qquad c_3 =(e,e^{n-1},e) \quad \dots \]
  \[ c_k = (e,e^{n-1},e,e^{n-1},e\dots ) \qquad \dots\]

  but there are significantly more weakly critical branching, for example the critical pair
  \[ e^{2n-1} = (e^{n})e^{n-1} = e^{n-1}(e^n) \]
  is a weakly critical $2$-branching, whose corresponding cell under \cref{cstr:cell_from_wcrit_branch} is $(e,e^{n-1},e^{n-1})$ and is a type $\II$ cell in the sense of Brown's \scheme as $e^{n-1}e^{n-1}$ has a reducible strict initial segment.

  More generally, a weakly critical $k$-branching necessarily starts with its first rule $f_1 = e^n$, but then $f_2$ is another $e^n$ that can be placed anywhere overlapping with $f_1$, which leaves $(n-1)$ possibilities (the overlap can be anything from $e^1$ to $e^{n-1}$), and $f_3$ can be anything overlapping with $f_2$, which leaves $(n-1)$ possibilities again, so in total we get $(n-1)^{k-1}$ different weakly critical $k$-branchings.

The corresponding weakly critical $(k+1)$-cells under \cref{cstr:cell_from_wcrit_branch} are all the $(k+1)$-cells of the form
 \[ (e,e^{n-1},e^{i_2},\dots,e^{i_k}) \]
 where each $0<i_j < n$.
  
\end{example}

\begin{definition}\label{def:weakly_crit_cell}
  We will say that an $n$-cells $(\omega_0,\dots,\omega_{n-1})$ of $N|G,R|$ is weakly critical if:
  \begin{itemize}
  \item $\omega_0$ is a single generator.
  \item For each $i>0$, there is a subrule of $\omega_0 \dots\omega_{n-1}$ that ends at the same place as $\omega_i$.
  \end{itemize}
\end{definition}

\begin{remark}
 Crucially, a weakly critical cell can be either type $\I$ if the rule that ends at the same place as $\omega_i$ is not included in $\omega_{i-1}\omega_i$ and it can be type $\II$ if this rule is included in $\omega_{i-1}\omega_i$ but there is another subrule earlier inside $\omega_{i-1} \omega_i$. Going back to \cref{ex:en=>1_weaklycrit}, the cell $(e,e^{n-1},e^2)$ is type $\II$ and the cell $(e,e^{n-1},e,e)$ is type $\I$ but both are weakly critical. They also corresponds to each other under the pariring function of Brown's \scheme.
\end{remark}

\begin{prop} For any reduced rewrite system, \cref{cstr:cell_from_wcrit_branch} induces a bijection between the set of weakly critical $n$-branchings in the sense of \cref{def:weakly_crit_branching} and \cref{lem:weakly_crit_branching} and the set of weakly critical $(n+1)$-cells in the sense of \cref{def:weakly_crit_cell}.\end{prop}

\begin{proof}
  From \cref{lem:weakly_crit_cell_normal} it is immediate that \cref{cstr:cell_from_wcrit_branch} always produce a weakly critical $(n+1)$-cell in the sense of \cref{def:weakly_crit_cell}. Moreover, the word $p$ can be recovered from $\omega_0\dots \omega_n$ and the endpoints of the $\omega_i$ for $i>0$ are the endpoints of the subrules $f_1,\dots,f_n$. As the rewrite system is reduced, any subrule is fully determined by its endpoint, so this shows that \cref{cstr:cell_from_wcrit_branch} is injective.

  Finally, given a weakly critical $(n+1)$-cell $(\omega_0,\dots,\omega_n)$ in the sense of \cref{def:weakly_crit_cell}, the word $p=\omega_0\dots \omega_{n}$, and let $f_i$ be the rules ending at the same place as $\omega_i$, which exists by \cref{def:weakly_crit_cell}. This defines an $n$-branching, we just need to check it is weakly critical, as its image by \cref{cstr:cell_from_wcrit_branch} will clearly be $(\omega_0,\dots,\omega_n)$. The rule $f_1$ needs to start at the beginning of $p$ because $\omega_0$ is just a single generator and so otherwise $f_1$ would be included in $\omega_1$ which would make $\omega_1$ not normal. Similarly if $f_i$ does not overlap with $f_{i-1}$ then $f_i$ would be included in $\omega_i$ which would make $\omega_i$ not normal. Finally $f_n$ ends at the end of $p$, so this concludes the proof.  
\end{proof}

\begin{prop}\label{prop:strongly_critical_branching} In a reduced rewrite system, a weakly critical branching $(p;f_1,\dots,f_n)$ corresponds under \cref{cstr:cell_from_wcrit_branch} to a critical cell if and only if for all $i$:

  \begin{enumerate}
  \item[$(\I_i)$] For $i>2$, $f_i$ does not overlap with $f_{i-2}$.
  \item[$(\II_2)$] $f_2$ is minimal, i.e. there are no subrules $r$ of $p$ such that $f_1 < r < f_2$.   
  \item[$(\II_i)$] For $i>2$, $f_i$ is minimal for condition $\I_i$. i.e. there are no subrules $r$ of $p$ such that $f_{i-1}<r<f_i$ and $r$ does not overlap with $f_{i-2}$.
  \end{enumerate}

Otherwise, if $i$ is the smallest index where either of these conditions fails, then $c(p;f_1,\dots,f_n)$ is type $\I$ if condition $\I_i$ fails and type $\II$ if $\I_i$ holds but $\II_i$ fails.
\end{prop}

These will be refers to as ``strongly critical branching'', what the proposition says is that our previous construction restrict to a correspondence between critical $n$-cells and strongly critical $(n+1)$-branching.

\begin{proof}
  These two conditions correspond to the direct translation of conditions $(C1_i)$ and $(C2_i)$ of \cref{def:Browm_collapsing}:
  \begin{itemize}
  \item A failure of condition $\I_i$ (for $i>2$) means that $f_i$ overlaps with $f_{i-2}$, which means that $f_i$ starts before $\omega_{i-1}$ and hence $\omega_{i-1}\omega_i$ is a strict subpath of $f_i$. This means that $\omega_{i-1}\omega_i$ is irreducible as the rewrite system is reduced (so no rules are strict subpaths of $f_i$), i.e. condition $(C1_i)$ fails. Conversely if condition $\I_i$ holds, then $f_i$ is a subrule of $\omega_{i-1}\omega_i$ which is hence reducible and condition $(C1_i)$ holds. Finally, note that conditions $(C1_1)$ and $(C1_2)$ can never fail for a weakly critical cell, as $\omega_0 \omega_1$ is the rule $f_1$ so is never normal, and the rule $f_2$ cannot include the first letter of $p$, as otherwise it would strictly contain the rule $f_1$, which is impossible, so $f_2$ is a subrule of $\omega_1\omega_2$, which is hence reducible.
  \item A failure of condition $\II_i$ (for $i \geqslant 2$) means there is a rule $r$ that ends before the end of $f_i$ (i.e. the end of $\omega_i$) but does not overlap with $f_{i-2}$ (if $i>2$), i.e. a subrule of $\omega_{i-1} \omega_{i}$ that ends before the end of $\omega_i$. This exactly means that $\omega_{i-1} \omega_i$ has a strict initial segment that is reducible, i.e. a failure of condition $(C2_i)$. Finally, note that condition $(C2_1)$ can never fail for a weakly critical cell in a reduced rewrite system: as $f_1 = \omega_0\omega_1$, there can never be a subrule of $\omega_0 \omega_1$ other than $f_1$ itself.
  \end{itemize}
 
\end{proof}

\begin{prop}
 The pairing $c$ between type $\I$ and type $\II$ cells of Brown's \scheme from \cref{def:Browm_collapsing} restricts to a bijection between weakly critical type $\I$ and weakly critical type $\II$ cells.
\end{prop}

\begin{proof}
  By the proof of \cref{prop:strongly_critical_branching}, given a weakly critical type $\II$ cell, corresponding to the weakly critical branching $(p;f_1,\dots,f_n)$ where the first failure of the conditions from \cref{prop:strongly_critical_branching} is a subrule $r$ of $p$ such that $f_{i-1}<r<f_i$ but $r$ does not overlap with $f_{i-2}$, the pairing from \cref{def:Browm_collapsing} will add an additional cut to the word $p$ at the end of the smallest such rule $r$, i.e. add the rule $r$ to our branching to give $f_1 < f_2 <\dots <f_{i-1} < r < f_i < \dots <f_n$. This is now a type $\I$ cell, as before $f_i$ conditions $\I_i$ and $\II_i$ always hold, and $f_{i-1} < r < f_i$ all overlap, i.e. condition $\I_{i+1}$ fails. Conversely, for any weakly critical type $\I$ cell $(\omega_0,\dots,\omega_{n-1})$, it corresponds by \cref{prop:strongly_critical_branching} to a weakly critical branching $(p;f_1,\dots,f_n)$ where the first failure of the conditions of \cref{prop:strongly_critical_branching} is that $f_{i-1}$ and $f_{i+1}$ overlap. As this is the first such failure, it means in particular that $f_i$ does not overlap with $f_{i-2}$ and $f_i$ is the smallest subrule of $p$ such that $f_{i-1} < f_i$ does not overlap with $f_{i-2}$. It follows that $(f_1,\dots,f_{i-1},f_{i+1},\dots,f_n)$ corresponds to a type $\II$ weakly critical cell whose associated type $\I$ cell, as in the previous half of the proof, would be $(f_1,\dots,f_{i-1},f_i,f_{i+1},\dots,f_n)$, which concludes the proof.\end{proof}

This immediately implies:

\begin{cor}\label{cor:weakly_crit_scheme}
  For a category $\Ccal$ presented by a reduced convergent rewrite system $(G,R)$, there is \scheme on $N(\Ccal)$ where the critical $n$-cells are the weakly critical $n$-cells of $N(\Ccal)$, i.e. corresponds to the weakly critical branching of $(G,R)$, and the rest of the structure is defined as in \cref{def:Browm_collapsing} for the other cells. 
\end{cor}

Another consequence of this characterization of strongly critical branchings given by \cref{prop:strongly_critical_branching}, as well as a good reason to use strongly critical branchings (or equivalently critical cells) compared to weakly critical ones, is that strongly critical branchings are often easier to list combinatorially:

\begin{cor}\label{cor:comb_strongly_critical} Given a reduced convergent rewrite system and $r_1,\dots,r_n,r_{n+1}$ a sequence of rewrite rules, then

  \begin{enumerate}
  \item There is at most one strongly critical branching whose rewrite rules are $r_1 < r_2 < \dots <r_n$.
  \item If such a strongly critical branching exists, call it $m$, then there will be one with rewrite rules $r_1 < \dots <r_n < r_{n+1}$ if and only if, given $w$ the part of $r_{n}$ at the end of $m$ that is not included in $r_{n-1}$, we can find an overlap between $w$ and $r_{n+1}$ in that order.
  \end{enumerate}
\end{cor}

A consequence of this proposition is that if a rewrite system has $n$ rewrite rules, then for $k>1$, Brown's \scheme has at most $n^{k-1}$ critical cells.

\begin{proof} This is proved by induction on $n$: item $(1)$ is clearly true for $n=1$; if item $(1)$ is true for some $n$, then item $(2)$ for $n$ follows immediately by the characterization of strongly critical branching given in \cref{prop:strongly_critical_branching}. And it also follows that the only way to make $r_1,\dots,r_{n+1}$ into a strongly critical branching is to place $r_{n+1}$ at the leftmost place at the end that overlaps with $r_n$ but not with $r_{n-1}$, hence this proves item $(2)$ for $n+1$, and by induction this establishes it for all $n$.
  
\end{proof}

\section{Examples and first applications}
\label{sec:example_main}

\begin{example}\label{ex:Free_graphs}
  Let $G$ be a graph. The rewrite system $(G,\varnothing)$ is trivially convergent and present the free category $G^*$ on the graph $G$, i.e. the category of paths in $G$. Our result applied to this rewrite system gives that we have an inner anodyne cofibration
  \[ N(G) \cto N(G^*) \]
  where $N(G^*)$ is the usual nerve of the category $G^*$ of paths in $G$ and $N(G)$ is the subsimplicial sets that only contains the $0$-cells (the vertex of $G$), the $1$-cells corresponding to edges of $G$, and all the degeneracies of the previous cells.

  That is, for any quasicategory $\Xcal$, we have an equivalence between the category $\Fun(G^*,\Xcal)$ of functors from $G^*$ to $\Xcal$ and the category $\Fun(N(G),\Xcal)$ whose objects can be described as morphisms of graphs from $G$ to the underlying graph of $\Xcal$ (i.e. the graph $\Xcal([1]) \rightrightarrows \Xcal([0])$).

  This can be seen as a special case of Proposition 5.25 of \cite{heuts2022simplicial} which we will recover in the following \cref{ex:Freely_adding_arrow}. A further special case of this example is that the spine inclusion $\text{Sp}[n] \to \Delta[n]$ is an inner anodyne cofibration.
\end{example}

\begin{example}\label{ex:Freely_adding_arrow}
  Let $\Ccal$ be an arbitrary category. Let $a,b$ be two objects of $\Ccal$. We consider the category $\Dcal = \Ccal[a \overset{g}{\to} b]$ obtained by freely adding an arrow from $a$ to $b$ to $\Ccal$.

  We can start from any convergent rewrite system for $\Ccal$, for example the standard rewrite system from \cref{ex:std_rewrite}, and add to it a single generator ``$g$'' from $a$ to $b$, which produces a presentation of $\Dcal = \Ccal[a \overset{g}{\to} b]$.

  As this new system has the exact same set of rules as the one for $\Ccal$, it is still convergent, and its critical cells are exactly the same as the critical cells of $\Ccal$ together with just one more $1$-dimensional critical cell corresponding to $g$. By \cref{lemma:Po_of_scheme} this proves that our pushout of categories:
\[\begin{tikzcd}
    \{0,1\} \ar[d,cto] \ar[r,"(a{,}b)"] \ar[dr,phantom,"\ulcorner"{description,very near end}]& \Ccal \ar[d,cto] \\
    (0 \to 1) \ar[r] & \Dcal 
  \end{tikzcd}\]
is also a homotopy pushout of quasicategories. This is exactly Proposition 5.25 of \cite{heuts2022simplicial}. We can do the exact same thing adding any set of new generators simultaneously.
\end{example}

This can be generalized even further:

\begin{example}
  Consider two categories $\Ccal_1$ and $\Ccal_2$, together with a set $S$ and two functions\footnote{We mean functions from $S$ to the sets of objects.} $S \to \Ccal_1$ and $S \to \Ccal_2$. Up to replacing the categories by equivalent ones, we can freely assume that these functions are injective. Starting from convergent rewrite systems $(G_1,R_1)$ and $(G_2,R_2)$ for $\Ccal_1$ and $\Ccal_2$, for example their standard rewrite system, we can form a convergent rewrite system for the pushout $\Ccal_1 \coprod_S \Ccal_2$: We simply take the ``pushout'' of the underlying graphs and the disjoint union of the sets of rewrite rules.

  The resulting rewrite system $(G_1 \coprod_S G_2, R_1 \coprod R_2)$ is convergent:

  \begin{itemize}
  \item It is terminating as an infinite sequence of reduction steps would need to have either an infinite number of steps from $R_1$ or from $R_2$; say it has an infinite number of steps in $R_1$. Our starting path looks like an alternating composite $a_1b_1\dots a_nb_n$ where the $a_i$ are paths in $G_1$ and the $b_i$ are paths in $G_2$. During our infinite sequence of steps, some of the $b_i$ might reduce to the identity, and some might not (for example, maybe they do not even have the same source and target!). But we can simply remove all the $b_i$ that reduce to identities at some point, and the sequence of all $R_1$ steps in the infinite sequence can be represented as acting directly on the resulting path where these $b_i$ have been removed, hence providing an infinite sequence of $R_1$-only reductions, which contradicts the fact that $(G_1,R_1)$ is terminating.
  \item It is confluent because no critical pairs can mix a rule in $R_1$ and a rule in $R_2$ as they never overlap. So all critical pairs are confluent, and we can conclude using Newman's \cref{lem:newman}.
  \end{itemize}

  Finally, since rules for $R_1$ and $R_2$ can never overlap, for cells of dimension $>0$, critical cells of $N(\Ccal_1 \coprod_S \Ccal_2)$ are just the critical cells of $\Ccal_1$ and $\Ccal_2$ taken separately, and as all $0$-dimensional cells are critical we recover the set $|\Ccal_1| \coprod_S |\Ccal_2|$ of objects. It follows that the pushout of $1$-categories

\[  \begin{tikzcd}
    S \ar[r] \ar[d] \ar[dr,phantom,"\ulcorner"{very near end,description}] & \Ccal_2 \ar[d] \\
   \Ccal_1 \ar[r]&  \Ccal_1 \coprod_S \Ccal_2
 \end{tikzcd}\]
is already a homotopy pushout of quasicategories. We have shown:
\end{example}

\begin{prop}\label{prop:pushout_of_1cat}
  A homotopy pushout of $\infty$-categories of the form:
  \[\begin{tikzcd}
    S \ar[r] \ar[d] & \Ccal_1 \ar[d] \\
    \Ccal_2 \ar[r] & P
  \end{tikzcd}\]
where $S$ is a set and $\Ccal_1$ and $\Ccal_2$ are $1$-categories, then the pushout $P$ is again a $1$-category.
\end{prop}

\begin{remark}
  The proposition can be alternatively phrased in terms of strict pushout of category with the assumption that the map $S \to \Ccal_1$ is injective on objects. It can also be deduced from \cref{thm:flat_rewrite} in the next section as explained in \cref{rk:flat_rewrite_imp_pushout_sets}.
\end{remark}

\begin{prop}\label{prop:adding_retract}
  Let $\Ccal$ be a $1$-category and let $i:A \to B$ be a monomorphism, and $t:A \to X$ an arrow in $\Ccal$. Let $\Dcal$ be the $\infty$-category which is generated from $\Ccal$ by freely adding an arrow $s:B \to X$ making the triangle commute:
\[
  \begin{tikzcd}
    A \ar[d,"i",hook] \ar[r,"t"] & X\\
    B \ar[ur,"s",dotted]
  \end{tikzcd}
\]
  i.e. the homotopy pushout:
  \[\begin{tikzcd}
    {\Lambda^0[2]} \ar[d,cto] \ar[r,"v"] & \Ccal \ar[d,cto] \\
    {\Delta[2]} \ar[r]& \Dcal 
  \end{tikzcd} \]
  where $v$ is the morphism sending $0 \to 1$ to $i$ and $0\to 2$ to $t$. Then $\Dcal$ is again a $1$-category.
\end{prop}

\begin{remark}
  Pushouts as in \cref{prop:adding_retract} are relevant for various theory of algebraic structures. Freely adding a map $s$ fitting in a triangle like this corresponds to ``freely adding an operation'' both in the setting of dependently typed algebraic theories and in monad-theory adjunction of Bourke and Garner \cite{bourke2018monads} (generalized to the $\infty$-categorical setting in \cite{henry2021higher}). We believe this proposition will be useful to prove coherence theorem in these setting - but we leave this to future work.
\end{remark}

\begin{proof}We can freely assume that $i$ is not invertible. We will first establish that $\Ccal$ admits a presentation by a rewriting system where $i$ is a generator, but $i$ is never the last\footnote{Remember that we are using diagrammatic composition order when talking about rewrite systems.} symbol of the domain of a rewrite rule. The underlying graph $G$ of our rewrite system is the subgraph of $\Ccal$ made of:
  \begin{enumerate}
  \item All objects of $\Ccal$.
  \item The arrow $i:A \to B$.
  \item All other arrows in $\Ccal$, except the arrows $f:X \to B$ that can be factored through $i:A \to B$.
  \end{enumerate}  
  As this is a subgraph of $\Ccal$, we have a map $G^* \to \Ccal$. We let $\Ncal \subseteq G^*$ (which will later be our set of normal paths) be the set of words of the form:
  \begin{enumerate}
  \item The elements of $G$ other than $i$, i.e. all arrows of $\Ccal$ not factoring through $i$.
  \item The words of the form $x i $ where $x:T \to A$.
  \end{enumerate}
  By definition, the composite $\Ncal \to G^* \to \Ccal$ is a bijection. We then set up rewrite rules so that $\Ncal$ is exactly the set of normal forms. For each pair $f,g$ of composable arrows in $G$, where $g$ is not $i$, we have a rule $[f][g] \Rightarrow x$ where $x$ is the unique element of $\Ncal$ that is equal to $fg$ in $\Ccal$. The normal paths are then exactly the single-character paths, and $f i$, like any other chain of two elements of $G$, is the domain of a rewrite rule. As we already have a bijection between normal forms and elements of $\Ccal$, we only need to check termination of the rewrite system to get confluence for free (as the end point of a chain of rewrites will always have to be the unique element of $\Ncal$ with the same image in $\Ccal$).

  For termination we need to distinguish three types of rewrite rules:
  \begin{enumerate}
  \item Rewrite rules of the form $[f] [g] \Rightarrow [r]$.
  \item Rewrite rules of the form $[f][g] \Rightarrow [h][i]$ where $f \neq i$.
  \item Rewrite rules of the form $[i][g] \Rightarrow [h] [i]$.
  \end{enumerate}

  If we assume toward a contradiction that we have an infinite sequence of rewrite rules, there can only be a finite number of rules of the first type as they decrease the length and no rule increases the length. After the last rule of the first type, there can only be a finite number of rules of the second type as they decrease the number of symbols not equal to $i$ and no other rule diminishes that number, and finally after the last rules for the first two types there can only be a finite number of rules of the last type as they are ``moving the $i$ to the right'', with everything else (the number of $i$'s and the length) being preserved, so there can only be a finite number of these as well.

 So our rewrite system is convergent, and it presents $\Ccal$. The $1$-category obtained by considering the pushout as in the proposition can be presented by the same rewrite system where we add an arrow $[r]: B \to X$ and a rule of the form $[i] [r] \Rightarrow [t]$. This new rewrite system is still terminating as we have added a rule that decreases the length to a system that had no rule increasing the length. The key point now is that because the original rewrite system for $\Ccal$ had no rewrite rule ending with $[i]$ and this one has one additional rule of the form $[i][r]$, there are no new critical pairs nor critical branchings in the new rewrite system. This immediately implies that this new rewrite system is convergent by Newman's \cref{lem:newman}. Note that the rewrite system for $\Dcal$ is clearly reduced as all rewrite rules have a domain of length $2$.

Now investigating the critical cells of $\Dcal$, in terms of the description of strongly critical branching given in \cref{cor:comb_strongly_critical}, we see that as the rewrite rule $[i] [r] \Rightarrow [t]$ has no overlap with the other critical cells from $\Ccal$, we see that the critical cells of $\Dcal$ are exactly the critical cells of $\Ccal$, the cell $[r]$ and the $2$-cell corresponding to the rewrite rule $[i] [r] \Rightarrow [t]$, and hence the map from $\Ccal \coprod_{\Lambda^0[2]} \Delta[2]$ to $\Dcal$ is anodyne, which proves that $\Dcal$ is the homotopy pushout described in the proposition.
\end{proof}

\begin{example}\label{ex:N2} The rewrite system for the monoid $\Nb^2$ from \cref{ex:Nk} is given by two generators $a,b$ and one rewrite rule $ba \Rightarrow ab$. There is no way to overlap $ba$ with itself, so there are no critical pairs or critical $n$-branchings for $n\geqslant 2$. It follows that the only critical cells of $N(B\Nb^2)$ are:

  \begin{enumerate}
  \item The unique $0$-cell $\bullet$.
  \item The two generators $(a)$ and $(b)$.
  \item The unique critical $2$-cell is $(b,a)$. The boundary of this cell can be drawn\footnote{Remember that in our convention, the path denoted $ab$ is the composite $b \circ a$.} as:
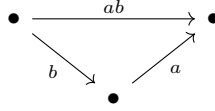
   \[ \begin{tikzcd}
      \bullet \ar[dr,"b"swap] \ar[rr,"ab"] & & \bullet \\
      & \bullet \ar[ur,"a"swap]  &
    \end{tikzcd} \]
    so essentially, it is a homotopy $b\circ a \simeq a \circ b$.
  \end{enumerate}

  So our main result shows that $N(B\Nb^2)$ is generated by these cells up to homotopy. For a quasicategory, a functor $N(B\Nb^2) \to \Xcal$ is, up to a contractible space of choices, the same as the data of an object $\bullet \in \Xcal$, two endomorphisms $a,b:\bullet \to \bullet $ in $\Xcal$ and a homotopy $ba \simeq ab$ in $\Xcal$.
\end{example}

\begin{example}\label{ex:N3} The rewrite system for the monoid $\Nb^3$ from \cref{ex:Nk} has generators $a,b,c$, and rewrite rules
  \[\{ ba \Rightarrow ab, cb \Rightarrow bc, ca \Rightarrow ac\}\]
  The critical cells are then given by:

  \begin{enumerate}
  \item The unique $0$-cell $\bullet$.
  \item The three generating $1$-cells $a$, $b$ and $c$.
  \item The three $2$-cells corresponding to rewrite rules $(b,a)$, $(c,b)$, $(c,a)$, whose boundary can be drawn as:
   
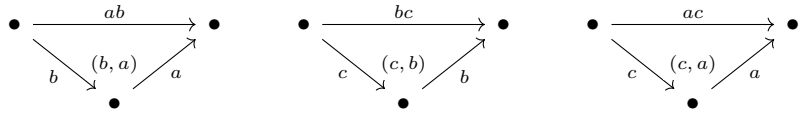
\[ \begin{tikzcd}
      \bullet \ar[dr,"b"swap] \ar[rr,"ab"] &\ar[d,phantom,"{(b,a)}"{description,font=\scriptsize}] & \bullet \\
      & \bullet \ar[ur,"a"swap]  &
    \end{tikzcd} \qquad  \begin{tikzcd}
      \bullet \ar[dr,"c"swap] \ar[rr,"bc"] & \ar[d,phantom,"{(c,b)}"{description,font=\scriptsize}] & \bullet \\
      & \bullet \ar[ur,"b"swap]  &
    \end{tikzcd} \qquad  \begin{tikzcd}
      \bullet \ar[dr,"c"swap] \ar[rr,"ac"] & \ar[d,phantom,"{(c,a)}"{description,font=\scriptsize}]  & \bullet \\
      & \bullet \ar[ur,"a"swap]  &
    \end{tikzcd}  \]
\item Finally, there is a unique critical $3$-cell $(c,b,a)$, corresponding to the unique critical pair $(cb)a =c(ba)$. Computing its boundary using the usual simplicial face operation and reduction in normal form gives us:

      
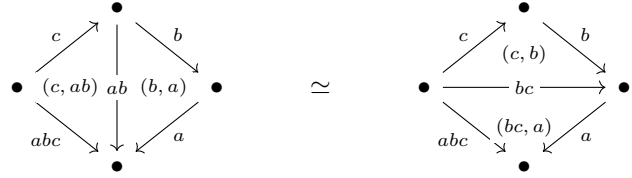
\[\begin{tikzcd}
    & \bullet \ar[dr,"b"] \ar[dd,"ab"description] &  & & &\bullet \ar[dr,"b"] & \\
    \bullet \ar[dr,"abc"swap] \ar[ru,"c"] & \ar[l,phantom,"{(c,ab)}"{description,font=\scriptsize}]  \ar[r,phantom,"{(b,a)}"{description,font=\scriptsize}] & \bullet \ar[dl,"a"]  & \simeq & \bullet \ar[rr,"bc"description] \ar[dr,"abc"swap] \ar[ru,"c"] &  \ar[u,phantom,"{(c,b)}"{description,font=\scriptsize}]  \ar[d,phantom,"{(bc,a)}"{description,font=\scriptsize}] & \bullet \ar[dl,"a"] \\
    & \bullet & & & & \bullet & 
  \end{tikzcd}\]
\end{enumerate}
Now, importantly, the boundary of this critical $3$-cell  involves some non-critical $2$-cell $(c,ab)$ and $(bc,a)$, both are type $\II$ cell, so we need to follow the process given by Brown's \scheme to understand how to further decompose these into critical cells. For example, $(c,ab)$ is associated to the type $\I$-cells $(c,a,b)$ whose boundary is

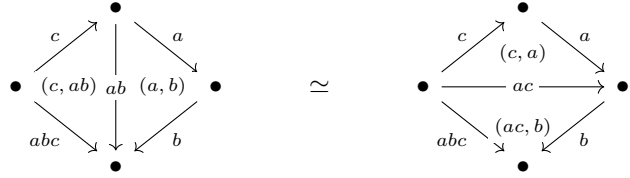
\[\begin{tikzcd}[ampersand replacement =]
    & \bullet \ar[dr,"a"] \ar[dd,"ab"description] &  & & &\bullet \ar[dr,"a"] & \\
    \bullet \ar[dr,"abc"swap] \ar[ru,"c"] & \ar[l,phantom,"{(c,ab)}"{description,font=\scriptsize}]  \ar[r,phantom,"{(a,b)}"{description,font=\scriptsize}] & \bullet \ar[dl,"b"]  & \simeq & \bullet \ar[rr,"ac"description] \ar[dr,"abc"swap] \ar[ru,"c"] &  \ar[u,phantom,"{(c,a)}"{description,font=\scriptsize}]  \ar[d,phantom,"{(ac,b)}"{description,font=\scriptsize}] & \bullet \ar[dl,"b"] \\
    & \bullet & & & & \bullet & 
  \end{tikzcd}\]
which expresses that $(c,ab)$ is homotopic to a composite of $(c,a)$ and $(ac,b)$ (the last $2$-cell $(a,b)$ is a type $\I$ cell expressing that $ab$ is just the composite of $a$ and $b$). $(c,a)$ is one of our critical cells, while $(ac,b)$ is another type $\II$ cell that is associated to the type $\I$ cell $(a,c,b)$ whose boundary includes the cells $(a,c)$,$(a,bc)$ that are type $\I$ (hence just express composition relation), the cell $(ac,b)$ we are trying to understand and finally the cell $(c,b)$ which is critical. So putting everything together, we can say informally that $(c,ab)$ is a homotopy $cab \simeq abc$ obtained by composing
\[ (ca)b \Rightarrow (ac)b = a (cb) \Rightarrow a (bc)\]
A similar discussion shows that $(bc,a)$ is a homotopy $bca \simeq abc$ obtained as a composite
\[ b(ca) \Rightarrow b (ac) = (ba)c \Rightarrow (ab)c \]

And so overall, the critical $3$-cell is a homotopy that expresses the usual ``hexagonal'' condition familiar for example from the definition of a braided or symmetric monoidal category:

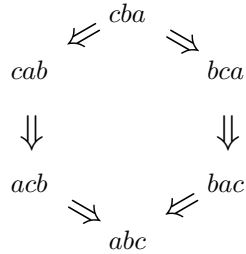
\[ \begin{tikzpicture}
\node (0) at (90:1.5) {$cba$};
\node (1) at (150:1.5) {$cab$};

\node (5) at (30:1.5) {$bca$};

\draw[draw = none] (0) -- (5) node[midway,sloped] {\scalebox{1.5}{ $\Rightarrow$}};

\draw[draw = none] (0) -- (1) node[midway,sloped] {\scalebox{1.5}{ $\Leftarrow$}};

\node (2) at (210:1.5) {$acb$};
\node (3) at (-90:1.5) {$abc$};

\node (4) at (-30:1.5) {$bac$};

\draw[draw = none] (1) -- (2) node[midway,sloped] {\scalebox{1.5}{ $\Rightarrow$}};
\draw[draw = none] (2) -- (3) node[midway,sloped] {\scalebox{1.5}{ $\Rightarrow$}};

\draw[draw = none] (5) -- (4) node[midway,sloped] {\scalebox{1.5}{ $\Rightarrow$}};
\draw[draw = none] (4) -- (3) node[midway,sloped] {\scalebox{1.5}{ $\Leftarrow$}};

\end{tikzpicture}
\]

So to give a functor $B\Nb^3 \to \Xcal$, for $\Xcal$ an $\infty$-category, it is enough - up to a contractible space of choices - to specify the image of the object $\bullet$, three endomorphisms $a,b,c$ of this object, three homotopies expressing the three commutation relations and finally one ``homotopy between homotopies'' expression that our three homotopies satisfy the hexagonal compatibility conditions from above. 
\end{example}

\begin{example}
  More generally, the rewrite system for $\Nb^k$ given in \cref{ex:Nk} has generators $a_1,\dots,a_k$ and its critical $n$-cells are exactly the
  \[ (a_{i_1},a_{i_2},\dots,a_{i_n})\]
  where $i_1 >i_2 > \dots >i_n$ are all in $\{1,\dots,k\}$.
  Unpacking their shape as we did before should in theory lead us to the formalism of permutahedra, but this seems difficult to do in a systematic way in the simplicial formalism.
\end{example}

\begin{example}
  The category $\Delta$ has a well-known presentation by generators and relations: it has objects $[n]$ for all $n \geqslant 0$, generating arrows:
  \[\delta^n_i:[n-1] \to [n] \qquad \sigma^n_i: [n+1] \to [n]\]
  with the so-called simplicial relations, which we write in the normal composition order, contrary to our general convention in the rest of the paper, to match usual practice:
\[\begin{array}{cccc}
    \forall i \leqslant j, & \delta^{n+1}_i \delta^n_j& \Rightarrow &\delta^{n+1}_{j+1} \delta^n_i \\
    \forall i \leqslant j, & \sigma^n_j \sigma^{n+1}_i & \Rightarrow& \sigma^n_i \sigma^{n+1}_{j+1} \\
    \forall i <j, & \sigma^n_j \delta^{n+1}_i &\Rightarrow& \delta^n_i \sigma^{n-1}_{j-1}\\
    \forall i=j \text{ or } i=j+1, & \sigma^n_j \delta^{n+1}_i&  \Rightarrow & 1_{[n]} \\
    \forall i>j+1, & \sigma^n_j \delta^{n+1}_i &\Rightarrow & \delta^n_{i-1} \sigma^{n-1}_{j}
\end{array}\]

The choice of orientation has been to produce the classical ``normal form'' for maps in $\Delta$ as:

\[ \delta_{i_1} \dots \delta_{i_k} \sigma_{j_1} \dots \sigma_{j_l}\]
where $i_1 >i_2 > \dots >i_k$ and $j_1 < j_2 <\dots < j_l$. It is easy to see that iteratively applying the rewrite rules always ends up at such a normal form, hence the rewrite system is terminating and confluence can be checked using Newman's lemma.

So we obtain a presentation of $\Delta$ by generators as an $\infty$-category.

Jardine's supercoherence theorem from \cite{jardine1991supercoherence} gives a description of a pseudo-functor $X:\Delta^\op \to \text{Cat}$ in terms of a sequence of objects $X([i])$, maps $X(\delta^n_i): X([n]) \to X([n-1])$ and $X(\sigma_i^n):[n] \to [n+1]$, natural transformations corresponding to the simplicial identities, and a sequence of 17 families of coherence equations relating these natural transformations.

It is possible to check that these 17 equations\footnote{The number 17 is somewhat arbitrary as it corresponds to 17 infinite families of equations. Fundamentally there are four types of critical pairs, depending on how many $\sigma$'s and $\delta$'s they involve; the rest of the subdivision depends on how we separate them into subfamilies according to the value of the indices involved, which involves some choices in how it is done.} correspond exactly to the critical pairs of the rewrite system above. In particular, our generalization of Brown's result - up to some calculation of the boundary maps that we are omitting - allows us to recover Jardine's theorem when applied to this rewrite system. For example, this immediately generalizes Jardine's result by replacing $\text{Cat}$ with an arbitrary weak $2$-category, by working in the underlying $(2,1)$-category and treating it as an $(\infty,1)$-category. In theory, this also provides a proof of an $n$-categorical version of the supercoherence theorem, with more coherence conditions in higher dimensions, but explicitly unpacking all the diagrams involved by hand seems like a daunting task - especially using the simplicial formalism.
\end{example}

The previous few examples show that it can be difficult to compute explicitly the real ``shape'' of the boundary of the critical cell in terms of the lower-dimensional critical cells. However, there are a lot of situations where we can apply our theorem very effectively by exploiting the fact that there are just not many, or not even any, critical $n$-cells. For example, the following is used in another work to appear: 

\begin{example}
  Consider the free $E_1$-monoid $X$ freely generated by elements $z,e_1,\dots,e_n$ and homotopies $z e_1 \simeq e_1 z, \dots, z e_n \simeq e_n z$. Then $X$ is $0$-truncated, i.e. corresponds to the ordinary monoid $\Nb \times F_n$ where $F_n$ is the free monoid on $n$ generators.
  
Indeed, consider the rewrite system with generators $z,e_1,\dots,e_n$ and rewrite rules $e_1 z \Rightarrow z e_1$, $\dots$, $e_n z \Rightarrow z e_n$. It is terminating as each application of a rule ``moves'' one instance of $z$ to the left, so the number of pairs of an $e_i$ appearing before a $z$ strictly decreases with each rule application, and there are no critical pairs as the domains of the rules cannot overlap, so confluence follows from Newman's \cref{lem:newman}. It also follows that the corresponding monoid $\Nb \times F_n$ is the free $\infty$-category with a single object (i.e. free $E_1$-monoid) generated by elements and homotopies as specified above, which proves the previous claim.  
\end{example}

\section{Applications to Dwyer maps}
\label{sec:Dwyer_maps}

In this section we show how rewriting methods can be used to give a new proof of the main result from \cite{hackney2024pushouts} about pushouts of Dwyer maps between categories being $(\infty,1)$-categorical pushouts. 
The idea is to show that Dwyer maps can be described by rewrite systems of a special form, for which pushouts are well behaved; namely, we will consider rewrite systems having the following properties:

\begin{assumption}\label{assum:flat_rewrite}
  Let $F:A \to B$ be a functor such that:
  \begin{enumerate}
  \item $A$ is presented by a convergent rewrite system $(G_A,R_A)$.
  \item $B$ is presented by a rewrite system $(G_B,R_B)$, such that $G_A \subseteq G_B$ and $R_A \subseteq R_B$, with $F$ being induced by this inclusion.
  \item Every rule in $R_B$ is either in $R_A$ or does not have any arrow in $G_A$ in its domain.
  \item The rewrite system obtained from the above by adding rules $f \Rightarrow 1$ for all $f \in G_A$, and identifying all the vertices in $G_A$, is terminating.
   \item $(G_B,R_B)$ is locally confluent (i.e. satisfies the assumption of Newman's \cref{lem:newman}).
  \end{enumerate}
\end{assumption}

\begin{lemma} \label{lem:flat_rewrite_terminante}
  Under \cref{assum:flat_rewrite}, the rewrite system $(G_B,R_B)$ for $B$ is also convergent.
\end{lemma}

\begin{proof}
  Assume toward a contradiction that there is an infinite sequence of rewrite steps in $(G_B,R_B)$. Assume also that this sequence has an infinite number of steps that are in $R_B-R_A$. Then as the domain of these rules only involves the symbols in $G_B-G_A$ and the other rules cannot affect these symbols, this can be immediately translated into an infinite sequence of rewrite steps for the modified rewrite system of \cref{assum:flat_rewrite}(4), which is contradictory. So there are only finitely many steps involving rules in $R_B - R_A$, and it follows that we have an infinite sequence of rewrite steps taken from $R_A$. Given such a sequence using only rules from $R_A$, the symbols not in $G_A$ are not affected, so we can see the initial path as a finite set of subpaths separated by symbols in $G_B - G_A$, and each rule applies to one of the subpaths in this set. So there will have to be an infinite number of steps applying to the same subpath, which contradicts the fact that $(G_A,R_A)$ is terminating. As \cref{assum:flat_rewrite} includes the fact that $(G_B,R_B)$ is locally confluent, this is enough to show it is convergent.
\end{proof}

\begin{theorem}\label{thm:flat_rewrite}
  Given a pushout square of $1$-categories
  \[ \begin{tikzcd}
      A \ar[d,"F"] \ar[r] & C \ar[d] \\
      B \ar[r] & D 
    \end{tikzcd} \]
  where $F:A \to B$ satisfies \cref{assum:flat_rewrite}, then it is also a pushout square of $\infty$-categories. 
\end{theorem}

\begin{proof}
  Pick a convergent rewrite system $(G_C,R_C)$ for $C$, for example the standard rewrite system as in \cref{ex:std_rewrite}. We can then construct a rewrite system $(G_D,R_D)$ for $D$ as follows:

  \begin{enumerate}
  \item Start from the rewrite system for $C$ (i.e. $G_D$ and $R_D$ contain $G_C$ and $R_C$).
  \item For each vertex in $G_B$ not in $G_A$, add a vertex to $G_D$.
  \item For each edge in $G_B$ not in $G_A$ add an edge to $G_D$ between the corresponding vertices (the one from the previous step for endpoints in $G_B-G_A$, and their image under the functor $A \to C$ for endpoints in $G_A$).
  \item Add all rewrite rules in $R_B - R_A$.
  \end{enumerate}

  This produces a rewrite system that presents $D$ (it is the pushout of presentations). We can immediately check that $(G_D,R_D)$ still satisfies all the assumptions of \cref{assum:flat_rewrite}, so in particular by \cref{lem:flat_rewrite_terminante} it is convergent.

  The key observation is that the critical cells of $N(D)$ for this rewrite system are all either critical cells of $B$ or of $C$. Indeed this is because the rules coming from $B$ and from $C$ can never overlap, as their domains involve distinct sets of symbols, and so it is impossible that a critical cell (seen as a critical branching as explained in \cref{sec:weak_vs_strong_criticallity}) involves rules from both sets at the same time.

  The result then immediately follows from \cref{cor:h_pushout_from_rw}: Brown's \schemes on $N(B)$ and $N(D)$, coming from their presentation by rewrite systems, restrict to \schemes on $N(A) \to N(B)$ and $N(C) \to N(D)$ by \cref{lem:extension_of_rw}; the previous ``key observation'' then shows that $N(B) \to N(D)$ induces a bijection between the critical cells of $N(A) \to N(B)$ and those of $N(C) \to N(D)$, as these only use the symbols from $G_B$ not in $G_A$.
  
\end{proof}

\begin{remark}\label{rk:flat_rewrite_imp_pushout_sets} We can immediately recover \cref{prop:pushout_of_1cat} from \cref{thm:flat_rewrite}: if $S$ is a set and $S \to \Ccal$ is an injection of $S$ into the set of objects of $\Ccal$, then the standard rewrite system for $\Ccal$ satisfies all of \cref{assum:flat_rewrite}.
\end{remark}

\begin{definition}
  A \emph{Dwyer map} is a full subcategory inclusion $F: \Ccal \subseteq \Dcal$ that factors as $\Ccal \subseteq \Wcal \subseteq \Dcal$ such that:
  \begin{enumerate}
  \item $\Ccal \to \Dcal$ is a sieve inclusion.
  \item $\Wcal \to \Dcal$ is a cosieve inclusion.
  \item $\Ccal \subseteq \Wcal$ is coreflective.
  \end{enumerate}
\end{definition}

\begin{lemma}\label{lem:Dwyer_rewrite}
  Given a Dwyer map $F: \Ccal \to \Dcal$, we can construct rewrite systems for $\Ccal$ and $\Dcal$ satisfying \cref{assum:flat_rewrite}.
\end{lemma}

\begin{proof}
  We can freely assume that $\Ccal$ is replete in $\Dcal$ (i.e. all objects isomorphic to an object of $\Ccal$ are in $\Ccal$). We write $\Acal$ for the full subcategory of $\Dcal$ of objects that are not in $\Ccal$ (and hence not isomorphic to objects of $\Ccal$ either).

We take the standard rewrite system on $\Ccal$. We extend it to a rewrite system for $\Dcal$ as follows:

  \begin{enumerate}
  \item We add all the generators and rules of the standard rewrite system for $\Acal$.
  \item For each $a \in \Wcal$ we add an arrow $p_a: c(a) \to a$, where $c(a)$ is the coreflection of $a$ on $\Ccal$.
  \item For each arrow $f:a \to a'$ for $a$ and $a'$ in $\Wcal$ we have a rule\footnote{Recall (\cref{notation:reverse_order}) that we use diagrammatic composition order when discussing paths in graphs.}
    \[ p_{a} f \Rightarrow c(f) p_{a'} \]
    corresponding to the commutative square:
    \[\begin{tikzcd}
      a \ar[r,"f"] & a' \\
      c(a) \ar[u,"p_a"] \ar[r,"c(f)"] & c(a') \ar[u,"p_{a'}"]
    \end{tikzcd}\]
\end{enumerate}

We check it satisfies \cref{assum:flat_rewrite}. Conditions (1) to (3) are clear (we will check that it presents $\Dcal$ later). For condition (4), if we add all the rules $c \Rightarrow 1$ for all $c \in \Ccal$, then the rule $p_{a} f \Rightarrow c(f) p_{a'}$ can be replaced by $p_{a} f \Rightarrow p_{a'}$, and all remaining rules decrease the length, so termination is immediate. For condition $(5)$, i.e. local confluence, the only critical pairs are the ones from the standard rewrite system on $\Ccal$ and $\Acal$, and the pair arising from the word $p_a f g$ for composable arrows $f:a \to a'$ and $g:a' \to a''$ of $\Wcal$, where local confluence corresponds to the following diagram:
  \[\begin{tikzcd}
      a \ar[r,"f"] & a' \ar[r,"g"]  & a'' \\
      c(a) \ar[u,"p_a"] \ar[r,"c(f)"] & c(a') \ar[u,"p_{a'}"] \ar[r,"c(g)"] &  c(a'') \ar[u,"p_{a''}"]
    \end{tikzcd}\]
  and the fact that $c(f)c(g) = c(fg)$. So we have checked all of \cref{assum:flat_rewrite}, to conclude we only need to check that this does present $\Dcal$.

  By \cref{lem:flat_rewrite_terminante}, the rewrite system above is convergent. The normal paths are exactly the arrows in $\Ccal$, the arrows in $\Acal$, and the arrows of the form $v p_a$ for $v$ an arrow in $\Ccal$. So between two objects of $\Ccal$ or between two objects of $\Acal$ we do recover exactly the same arrows as in $\Dcal$. The only other arrows in $\Dcal$ are the arrows from an object of $\Ccal$ to an object of $\Wcal$ (indeed as $\Ccal$ is a sieve there are no other arrows into $\Ccal$, and as $\Ccal \subseteq \Wcal$ and $\Wcal$ is a cosieve, there are no arrows $c \to x$ for $c \in \Ccal$ and $x \notin \Wcal$). But any such arrow $c \to a$ is written uniquely in the form $v p_a$ by coreflectivity of $\Ccal$ in $\Wcal$, so we also have a correspondence between these arrows and the normal forms. We only need to check that all these arrows compose as expected, but this is a straightforward verification.
\end{proof}

Combining \cref{thm:flat_rewrite} and \cref{lem:Dwyer_rewrite}, we recover the main result from \cite{hackney2024pushouts}

\begin{cor}
  Any pushout of $1$-categories of the form
  \[\begin{tikzcd}
      A \ar[d,"F"swap] \ar[r] & C \ar[d] \\
      B \ar[r] & D 
    \end{tikzcd} \]
where $F$ is a Dwyer map is also a pushout of $\infty$-categories.  
\end{cor}

\section{Applications to coherence for diagrams}
\label{sec:example_diag}

The goal of this section is to apply rewriting methods to establish a coherence theorem for planar loop-free diagrams. The type of diagram we consider is explained in \cref{assum:good_graph} just below, and we will establish that for this type of diagram, the $\infty$-category they present is a $1$-category. In particular, this allows us to quickly justify many operations of pasting and gluing of such diagrams that we often do in higher category theory, as explained in \cref{rk:gluing_diagrams}.

There seems to be some intersection between this coherence result and the pasting theorem for $(\infty,2)$-categories established in \cite{hackney2023pasting}, at least in the sense that both allow one to do something like the pasting and gluing construction we discuss in \cref{rk:gluing_diagrams}. Whether there is a more precise connection between the two results remains unclear to the author. Our result applies to more diagrams, but the pasting theorem of  \cite{hackney2023pasting} says something about the $(\infty,2)$-category presented by the diagram.

\begin{assumption}\label{assum:good_graph}
  In this section, we consider graphs $G$ such that:

  \begin{enumerate}
  \item $G$ is a finite directed graph.
  \item $G$ has no oriented loops.
  \item $G$ is a planar graph, in the sense that it is explicitly (and smoothly outside of vertices) embedded in the plane. In particular, by the Jordan curve theorem, the complement of the graph in the plane decomposes into a finite set of \emph{bounded} components, which we call inner faces, and one unbounded component (the outer face).
  \item \label{assum:good_graph:sink_source} The cycle boundary of each inner (i.e. bounded) face of $G$ has a unique ``source'' and a unique ``sink''. That is, the boundary of the face decomposes into two oriented paths, both starting at the same point (the source) and ending at the same point (the sink). For such a face $R$, one of these oriented paths goes in the clockwise direction around the face and is denoted $\partial^- R$, and the other goes in the counter-clockwise direction and is denoted $\partial^+ R$. 
    
  \end{enumerate}

The notion of planar graph can be defined purely combinatorially - see for example \cite{diestel2012graph} or \cite{mohar2001graphs} for basic references - but if the reader is happy relying on well-known topological facts like the Jordan curve theorem to formalize intuitions related to inner and outer faces, this combinatorial theory shouldn't be necessary for our purposes; it is only needed if one wants a purely combinatorial description.
\end{assumption}

\begin{prop}\label{prop:graphs_rewrite}
Given a graph as in \cref{assum:good_graph}, the corresponding rewrite system is convergent and has no critical branching.
\end{prop}

\begin{cor}\label{cor:graph_rewrite}
  Given a graph $G$ as in \cref{assum:good_graph}, the $\infty$-category generated by the graph $G$ together with one homotopy $\gamma_R: \partial^-R \simeq \partial^+ R$ for each face $R$ is equivalent to the $1$-category with the same presentation. 
\end{cor}

\begin{proof}
  Each edge $f$ in the graph has a face to its left and a face to its right (possibly the same if we allow the situation as in \cref{rk:eq_diagram}). And given a face $R$, the edges that are in $\partial^-R$ have $R$ to their right and the edges in $\partial^+ R$ have $R$ to their left. In particular, each edge appears in the domain of at most one rewrite rule. So no rewrite rule can overlap and hence there are no critical pairs or critical branchings. So, by Newman's \cref{lem:newman} we only have to prove termination to conclude the proof of the proposition.

  Informally, the proof of termination works because each application of a rule ``decreases the area under the path''; since there are finitely many paths in the graph (thanks to the assumption that the graph is loop-free), the process cannot cycle and hence has to terminate.

  Let us make this a little more formal: as $G$ has no directed loop, no directed path can pass twice through the same vertex, so as $G$ is also finite there are finitely many paths in $G$. To each directed path $p$ in the graph, we can associate:
  \[ \Phi(p) = \int_p x dy = \int_{t \in [0,1]} x(t) y'(t) dt \in \Rb\]
  where $(x(t),y(t))$ is a smooth parametrization of $p$, but the integral does not depend on the choice of parametrization. $\Phi(p)$ is essentially a signed version of the ``surface under the curve $p$'', where the sign depends on whether $p$ is under or over the $x=0$ axis and on whether it is going ``left to right'' or ``right to left''. If $p \Rightarrow q$ is a one-step reduction corresponding to a rule $R: \partial^-R \Rightarrow \partial^+ R$ then:

  \[ \Phi(q) = \Phi(p) - \left( \int_{\partial^- R} x dy - \int_{\partial^+ R} x dy\right) = \Phi(p) - \int_{\partial R} x dy  \]
  where $\partial R$ is the boundary of the face $R$ oriented counterclockwise. By Green's theorem, this integral is exactly the area of $R$, hence $ \Phi(q) = \Phi(p) - \text{Area}(R)$. So each application of a rule reduces $\Phi$ by at least the area of the smallest face; as there are only finitely many paths, i.e. finitely many possible values of $\Phi$, this proves that there cannot be an infinite sequence of reduction steps, hence the rewrite system is terminating.

The proof can also be rephrased in a purely combinatorial way by replacing $\Phi$ with a (signed) count of the number of faces between $\gamma_0$ and $\gamma$ when $\gamma_0 \Rightarrow^* \gamma$, and checking that this strictly increases (or decreases, depending on convention) by $1$ every time we apply a reduction rule. Making this formal requires slightly more graph theory to make sure this notion makes sense, but it is definitely possible.  
\end{proof}

\begin{remark}\label{rk:gluing_diagrams}
  The main application of \cref{prop:graphs_rewrite} and \cref{cor:graph_rewrite} is to give formal justification to many constructions of ``cutting and pasting'' of diagrams that we very often do in higher category theory without always properly justifying them - even if each instance of these manipulations is not very hard to prove rigorously, the proofs are not always the same and are rarely included in published papers. We believe these results are a very effective way to cleanly justify all these manipulations with almost no effort. To give an example of what we mean, looking at a complicated ``commutative diagram'' like:
\[\begin{tikzcd}
{} \arrow[dd] \arrow[r] & {} \arrow[d] \arrow[r]  & {} \arrow[d] \arrow[rd] &              \\
                        & {} \arrow[r] \arrow[ld] & {} \arrow[r] \arrow[d]  & {} \arrow[d] \\
{} \arrow[r]            & {} \arrow[r]            & {} \arrow[r]            & {}          
\end{tikzcd}\]
it clearly satisfies \cref{assum:good_graph}, so by \cref{cor:graph_rewrite}, saying that such a diagram commutes simply means that each face commutes. So in order to build a diagram like this in a quasicategory $\Ccal$, we can separately justify that each face commutes without ever having to address the diagram as a whole - and we can use a different argument for each face, coming from a different commutative diagram. Once we know the whole diagram commutes in a category $\Ccal$, we conclude that we have a morphism of simplicial sets $N(P) \to \Ccal$ (where $P$ is the category, or more generally the poset, presented by the diagram), so we can freely extract parts of the diagram that are again commutative and glue them to other diagrams.
\end{remark}

\begin{remark}\label{rk:eq_diagram} In \cref{assum:good_graph}.\ref{assum:good_graph:sink_source} we can allow faces of the form:
\[\begin{tikzcd}
A \arrow[rr, "b", bend left] \arrow[rr, "c"', bend right] & X \arrow[l, "a" description] & B
\end{tikzcd}\]
More formally, we mean that when an edge has the same face to its left and its right (like $a$ here), it is counted as appearing twice in the boundary, so the point $X$ has indeed two edges starting from it ($a$ twice), so it is a source; $A$ has the same number of incoming and outgoing edges that are part of the boundary of the face, and $B$ has two incoming edges, so it is a sink. Of course, when looking for sinks and sources we only count the edges that are part of the boundary, not any other edges in the graph. In the case of the face above we have
\[ \partial^- R  = b \circ a \qquad \partial^+ R = c \circ a\]
so the only equation imposed in the rewrite system is $b\circ a = c \circ a$; the equation $b =c $ is \emph{not} imposed.

The only modification to the proof of \cref{prop:graphs_rewrite} is that when forming $\left( \int_{\partial^- R} x dy - \int_{\partial^+ R} x dy\right)$, the edges that are ``inside'' the face, like $a$ in our example, i.e. that have the face $R$ both to their left and to their right, appear in both $\partial^- R$ and $\partial^+ R$, so they cancel out and we obtain the integral over the real ``outer boundary'' of the face, which in the end gives the area, ignoring the inner edges.
\end{remark}

\begin{remark}
  Also, the proof of \cref{prop:graphs_rewrite} and \cref{cor:graph_rewrite} still works if we remove item $(4)$ from \cref{assum:good_graph} above, but only putting in a rewrite rule (and homotopies in \cref{cor:graph_rewrite}) for the faces that actually satisfy $(4)$ and leaving the other ones ``open''. More generally, we can choose any subset of faces that satisfies condition $(4)$ and only put homotopies and rewrite rules for these. 
\end{remark}

\begin{example}\label{ex:cube_1}
  The reader should be aware that:
  \begin{enumerate}
  \item not all diagrams satisfying \cref{assum:good_graph} are ``commutative'', i.e. present a poset.
  \item The category presented by this rewrite system depends on ``how the diagram is drawn'', that is, on the choice of the planar embedding.
  \end{enumerate}
   The example from \cref{rk:eq_diagram} is one such example of the first point, and moving the $X$ ``outside'' of the face produces an example of the second point, as with $X$ outside the face, it becomes an equation $b=c$. For a slightly different type of example, consider the cubical diagrams drawn in the plane in two different ways as:
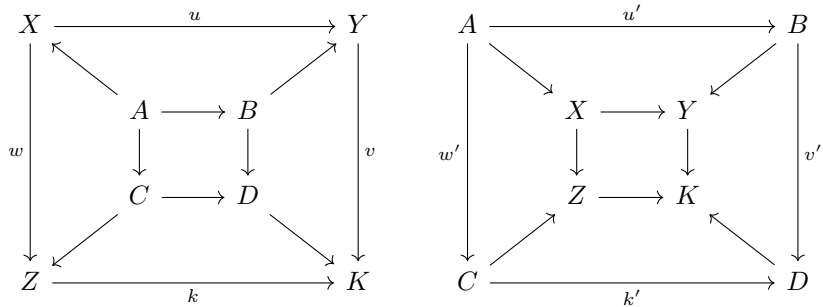
\[\begin{tikzcd}
X \arrow[ddd, "w"'] \arrow[rrr, "u"] &                                  &                        & Y \arrow[ddd, "v"] & A \arrow[ddd, "w'"'] \arrow[rrr, "u'"] \arrow[rd] &                       &             & B \arrow[ddd, "v'"] \arrow[ld] \\
                                     & A \arrow[d] \arrow[r] \arrow[lu] & B \arrow[d] \arrow[ru] &                    &                                                   & X \arrow[r] \arrow[d] & Y \arrow[d] &                                \\
                                     & C \arrow[r] \arrow[ld]           & D \arrow[rd]           &                    &                                                   & Z \arrow[r]           & K           &                                \\
Z \arrow[rrr, "k"']                  &                                  &                        & K                  & C \arrow[rrr, "k'"'] \arrow[ru]                   &                       &             & D \arrow[lu]                  
\end{tikzcd}\]

Then in the left diagram, there is no equation $uv = wk$ as this outer square is not a face, and cannot be obtained by combining other faces (in fact, both $uv$ and $wk$ are already normal). On the right we have a different embedding of the same diagram, and now the equation $uv = wk$ is present, but it is the equation $u' v'=w'k'$ that is missing.
\end{example}

\begin{remark}
  While \cref{prop:graphs_rewrite} really relies on the graphs being planar, the general idea of seeing a diagram as a rewrite system with no or few critical branchings works with higher-dimensional diagrams as well. Given a cubical diagram like

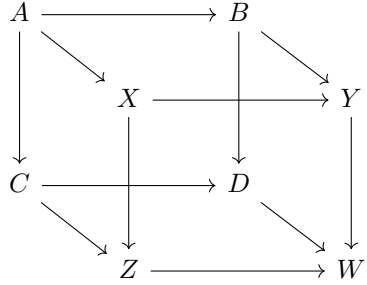
\[\begin{tikzcd}
A \arrow[rr] \arrow[dd] \arrow[rd] &                         & B \arrow[rd] \arrow[dd] &              \\
                                   & X \arrow[rr] \arrow[dd] &                         & Y \arrow[dd] \\
C \arrow[rd] \arrow[rr]            &                         & D \arrow[rd]            &              \\
                                   & Z \arrow[rr]            &                         & W           
                                 \end{tikzcd}\]
                               It is possible to orient the faces so as to obtain a convergent rewrite system with only one critical branching (start from \cref{ex:cube_1} and add one rewrite rule for the outer face - the termination argument still works, and it adds a single critical pair). Unpacking our result in this case recovers the expected result that a cubical diagram in an $\infty$-category is the same as 6 squares commuting up to homotopies together with one ``hexagonal'' compatibility relation between the 6 homotopies.

In particular, arguments like the one in \cref{rk:gluing_diagrams} can be applied to higher-dimensional diagrams as well, as long as we are careful with the critical cells of higher dimension as well. Developing a general theory of these does not seem convenient, so we leave that aside for now.
\end{remark}

\bibliography{../../../Biblio}{}
\bibliographystyle{alpha}

\end{document}